\documentclass[11pt,a4paper]{article}

\usepackage[utf8]{inputenc}
\usepackage[T1]{fontenc}
\usepackage{amsmath,amssymb,amsthm}
\usepackage{mathrsfs}
\usepackage{lmodern}
\usepackage{microtype}

\usepackage{graphicx}
\usepackage{xcolor}
\usepackage{subcaption}
\usepackage{array,multirow,tabularx}

\usepackage{geometry}

\usepackage{authblk}

\usepackage[all]{xypic}
\usepackage{nicefrac}
\usepackage{algorithm}
\usepackage{algpseudocode}
\usepackage{comment}
\usepackage{enumitem}

\usepackage{hyperref}
\hypersetup{
  colorlinks,
  citecolor=red,
  linkcolor=blue,
  urlcolor=green}

\theoremstyle{plain}
\newtheorem{theo}{Theorem}[section]
\newtheorem{prop}[theo]{Proposition}
\newtheorem{lemm}[theo]{Lemma}

\theoremstyle{remark}
\newtheorem{rema}[theo]{Remark}

\theoremstyle{plain}

\newtheorem{problem}[theo]{Problem}
\newtheorem{definition}[theo]{Definition}

\theoremstyle{definition}
\newtheorem{assump}{Assumption}

\def\R{\mathbb{R}}

\newcommand{\norm}[1]{\left\lVert#1\right\rVert}
\newcommand{\sclr}[1]{\left\langle#1\right\rangle}

\newcommand{\prox}[1]{{\rm Prox}_{#1}}
\newcommand{\proj}[1]{{\rm Proj}_{#1}}
\newcommand{\multimapto}{\mathrel{\text{\raisebox{0.5ex}{$\rightarrow$}\hspace{-1.0em}\raisebox{-0.5ex}{$\rightarrow$}}}}
\DeclareMathOperator*\argmin{\textrm{arg\,min}}
\newcommand{\bfcalH}{\mathcal{H}}

\newcommand{\ddt}{\,{\rm d}t}
\newcommand{\ddl}{\,{\rm d}\lambda}
\newcommand{\dds}{\,{\rm d}s}
\newcommand{\ddtau}{\,{\rm d}\tau}
\newcommand{\ddD}{\,{\rm D}}
\newcommand{\DRS}{{\rm DRS}}

\title{Douglas-Rachford splitting method for nonconvex stochastic optimal control}

\author{\makebox[\dimexpr\textwidth-2\tabcolsep\relax]{%
  \hfill François~Caffier\thanks{\scriptsize IFP Energies nouvelles, Rueil Malmaison, 92063, France; Université Paris-Saclay, CNRS, CentraleSupélec, Inria, Laboratoire des signaux et systèmes, 91190, Gif-sur-Yvette, France. {\tt francois.caffier@ifpen.fr}}%
  \hfill Paul~Malisani\thanks{\scriptsize IFP Energies nouvelles, Rueil Malmaison, 92063, France. {\tt paul.malisani@ifpen.fr}}%
  \hfill Laurent~Pfeiffer\thanks{\scriptsize Unité de Mathématiques Appliquées, ENSTA, Institut Polytechnique de Paris, 91120 Palaiseau, France. {\tt laurent.pfeiffer@ensta.fr}}%
  \hfill}}

\date{}

\begin{document}

\maketitle

\begin{abstract}
Motivated by the numerical resolution of stochastic optimal control problems, particularly those involving nonlinear dynamics, we investigate the convergence properties of the Douglas-Rachford splitting algorithm in a nonconvex and infinite-dimensional setting. We introduce an extension of the method that applies when one component of the objective function is continuously differentiable with a locally Lipschitz gradient, while the other is weakly lower semicontinuous with a bounded domain.
\end{abstract}


\newpage 

\tableofcontents

\section{Introduction}

\label{sec:introduction}

This work is concerned with the numerical resolution of nonconvex, continuous-time, stochastic optimal control problems. Our general objective is to design and analyze a tractable splitting method for its resolution, such that each iteration requires solving a family of independent low-dimensional optimal control problems.

In this section, we first introduce the notations used throughout the rest of the document. In Subsection \ref{subsec:soc}, we introduce our stochastic problem of interest, Problem~\ref{pb:problem_for_pha}. We propose a slightly more abstract formulation of the problem (Problem \ref{pb:problem_optimal_control}) that will be more convenient to handle. Subsection~\ref{subsec:DRS_method_review} introduces the Douglas-Rachford Splitting method, the numerical approach adopted to solve this problem, and reviews its current state of the art. Our contributions are summarized in Subsection~\ref{sec:main_result}.

\subsection{Notations}

\label{subsec:notations}

In this paper, we consider a Hilbert space $H$ and a function $\phi : H \to \bar{\mathbb{R}}$, where $\bar{\mathbb{R}} := \mathbb{R} \cup \{+\infty\}$ is the extended real line. We write $0_H$ for the null element of $H$. For any $R > 0$, we denote by $\bar{B}(R)$ the closed ball of radius $R$ centered at the origin. Given a set $E \subset H$, the \emph{indicator function} $\iota_E : H \to \bar{\R}$ is defined by
\[
\iota_E(x) := \begin{cases}
0 & \text{if } x \in E, \\
+\infty & \text{otherwise}.
\end{cases}
\]
Given $x \in H$, we define the \emph{distance to the set} $E$ as \( \mathrm{\mathbf{dist}}(x,E) := \inf_{y \in E} \|x - y\|. \)

\paragraph*{Convexity} A function $\phi \colon H \rightarrow \bar{\R}$ is said to be \emph{proper} if it is not identically equal to $+\infty$. Its \emph{domain} is defined as
\( \mathrm{\mathbf{dom}}(\phi) := \{ x \in H \mid \phi(x) < +\infty \} \),
and its \emph{epigraph} is the subset of $H \times \mathbb{R}$ given by \(
\mathrm{\mathbf{epi}}(\phi) := \{ (x, r) \in H \times \mathbb{R} \mid \phi(x) \leq r \}.
\)
We say that $\phi$ is \emph{lower semicontinuous} (lsc) if its epigraph is closed in $H \times \mathbb{R}$, and \emph{convex} if its epigraph is convex. We denote by $\Gamma_0(H)$ the set of proper, convex, and lower semicontinuous functions from $H$ to $\bar{\mathbb{R}}$. Finally, $\phi$ is said to be \emph{strongly convex} if there exists $\alpha>0$ such that $\phi - \nicefrac{\alpha}{2} \norm{\cdot}^2$ is convex.

\paragraph*{Lebesgue spaces} We endow $\R^n$ with an inner product denoted by $\sclr{.,.}$, and the associated norm $\norm{.}=\sqrt{\sclr{.,.}}$. Let $T>0$ be a finite time horizon. A particular case of Hilbert space we are interested in is $\mathcal{H}:=L^2(0,T;\R^m)$, the space of square integrable functions from $(0,T)$ to $\mathbb{R}^m$, endowed with the inner product
$\sclr{u_1,u_2} := \int_0^T \sclr{u_1(t),u_2(t)} \ddt$,
and corresponding norm $\norm{.}_{2}$. We denote by $L^\infty(0,T;\mathbb{R}^m)$ the Lebesgue space of almost everywhere (a.e.) bounded functions from the interval $(0,T)$ to $\mathbb{R}^m$, equipped with its standard norm $\|u\|_\infty$. A sequence of functionals $(u_k)_{k \in \mathbb{N}}$ in this space is said to be \emph{uniformly bounded} if there exists a constant $M > 0$ such that $\|u_k\|_\infty \leq M$ for all $k \in \mathbb{N}$.

Let $\tau= (\tau_j)_{j=0,\ldots,J}$ denote a sequence of times such that
$0=\tau_0 < \tau_1 < \ldots
< \tau_{J}= T$. 
We say that $u \in \mathcal{H}$ is piecewise $\nicefrac{1}{2}$-H\"older continuous with respect to $\tau$ if there exists $M\geq 0$ such that
\begin{equation*}
\| u(t_2)- u(t_1) \|
\leq M \, |t_2-t_1|^{1/2},
\quad \forall j=0,\ldots,J-1,\
\text{for a.e. } t_1, t_2 \in (\tau_j,\tau_{j+1}).
\end{equation*}
In this case $M$ is referred to as a modulus and we denote as $C^{0,\nicefrac{1}{2}}_{\tau}(0,T;\R^m)$ the set of piecewise $\nicefrac{1}{2}$-H\"older continuous with respect to $\tau$.

\paragraph*{Differentiability}
In this article, differentiability on $H$ refers to Fr{\'e}chet differentiability. The class of functions $\varphi$ that are continuously differentiable on $H$ is denoted by $C^1(H)$. We denote by $C^{1,1}(H)$ (respectively, $C^{1,1}_{\text{loc}}(H)$) the set of functions $\varphi \colon H \rightarrow \mathbb{R}$ that are continuously differentiable and whose gradient $\nabla \varphi$ is Lipschitz continuous (respectively, locally Lipschitz continuous), that is, Lipschitz on $H$ (respectively, on every bounded subset of $H$). We say that $\varphi \in C^{1,1}(H)$ is $L$-\textit{smooth} if its gradient is $L$-Lipschitz continuous.

The \emph{generalized subdifferential} of $\varphi$ at a point $x \in \mathrm{\mathbf{dom}}(\varphi)$ is defined as
\[
\hat{\partial} \varphi(x) := \left\{ v \in H \mid \liminf_{y \to x} \frac{\varphi(y) - \varphi(x) - \langle v, y - x \rangle}{\|y - x\|} \geq 0 \right\}.
\]
If $\varphi$ is differentiable at $x$, then the gradient $\nabla \varphi(x)$ coincides with the unique element of the generalized subdifferential $\hat{\partial} \varphi(x)$. When $\varphi \in \Gamma_0(H)$, this reduces to the \emph{convex subdifferential}, given by
\[
\partial \varphi(x) := \left\{ v \in H \mid \forall y \in H, \ \varphi(y) \geq \varphi(x) + \langle v, y - x \rangle \right\}.
\]
A point $x \in H$ is a \emph{stationary point} of $\varphi$ if $0 \in \hat{\partial} \varphi(x)$.

\paragraph*{Proximal operator}

The \textit{proximal operator} of a given function $\varphi : H \rightarrow \bar{\R}$, with parameter $\gamma > 0$, is denoted as $\prox{\gamma \varphi} : H \multimapto H$ and defined as
\begin{equation}
\label{eq:problem_prox}
\prox{\gamma \varphi}(x) := \argmin_{y \in H} \left\{ \varphi(y) + \frac{1}{2 \gamma} \norm{y-x}^2 \right\}.
\end{equation}
Without any additional assumption, $\prox{\gamma \varphi}$ may be empty or may contain several points.
When $\varphi \in \Gamma_0(H)$, the prox operator is single-valued. In this case we consider it as a function from $H$ to $H$ (and not anymore as a set-valued mapping).

The \textit{reflected operator} of $\varphi$, with parameter $\gamma > 0$, is denoted as $R_{\gamma \varphi} : H \multimapto H$ and defined as
\begin{equation}
\label{eq:def_reflected}
R_{\gamma \varphi}(x)
= 2\prox{\gamma \varphi}(x)- \mathrm{Id}
= \Big\{ 2y-x \,|\, y \in \prox{\gamma \varphi}(x) \Big\}.
\end{equation}
When $\varphi$ lies in $\Gamma_0(H)$, the reflected operator is single-valued and treated as a mapping from $H$ to $H$.

\subsection{ Stochastic optimal control and scenario decomposition}
\label{subsec:soc}

\paragraph*{Problem formulation}

Our main problem of interest, Problem~\ref{pb:problem_for_pha}, is a stochastic optimal control problem with continuous-time dynamics and finite sets of scenario with a tree structure. Let us describe it precisely. The problem is formulated over the interval $[0,T]$ and involves a sequence of intermediate times $0=\tau_0 < \tau_1 < \ldots < \tau_{J-1} < \tau_J=T$.
We first fix a finite set $\Xi$ of random events $(\xi^s)_{s=1,\ldots,S}$, each occurring with probability $p_s>0$. We refer to an element $\xi^s$ as a scenario and assume that it takes the form of vector $(\xi^s(\tau_1),\ldots,\xi^s(\tau_{J-1})) \in \Omega^{J-1}$, where $\Omega$ is a given set.
Given $t \in (\tau_j,\tau_{j+1})$,
we say that two scenarios $\xi^s$ and $\xi^{s'}$ are \emph{indiscernible} at time $t$ if $\xi^s(\tau_k)= \xi^{s'}(\tau_k)$, $\forall k=1,\ldots,j$.
The control variable is a function $u \in L^2\left((0,T)\times \Xi; \R^M \right)$. We say that a control $u$ is \emph{non-anticipative} if the following holds for almost every $t \in (0,T)$:
\begin{equation}
\label{eq:non-anticipativity}
\forall \xi^s, \xi^{s'} \in \Xi,\quad
\text{$\xi^s$ and $\xi^{s'}$ are indiscernible at time $t$}
\Longrightarrow
\forall \tau \in [0,t], \ u(\tau,\xi^{s})=u(\tau,\xi^{s'}).
\end{equation}
This condition reflects the fact that decisions taken at time $t$ depend only on the revealed portion of the scenario. We denote as $\mathcal{N}$ the vector subspace of non-anticipative controls.

\begin{problem}
\label{pb:problem_for_pha}
Given an initial condition $x_0 \in \R^N$, a control set $U \subset \R^M$, a running cost $\ell \colon U \times \R^N \times \Xi \rightarrow \R$, a final cost $g \colon \R^N \times \Xi \rightarrow \R$, and a dynamics $f \colon U \times \R^N \times \Xi \rightarrow \R^N$, we consider
\begin{equation*}
\label{eq:pb_stochastic_multistage}
\inf_{u \in \mathcal{U} \cap \mathcal{N}}
\mathbb{E}_{\xi}
\!\left[
\int_{0}^{T} \ell\!\left(u(t,\xi), x[u](t,\xi), \xi \right) \ddt
+ g \left(x[u](T,\xi), \xi \right)
\right],
\end{equation*}
where $\mathcal{U}$ is the set of controls in $L^2\left((0,T)\times \Xi; \R^M \right)$ taking values in $U$ and where $x[u]$ denotes the solution to
\begin{equation}
\label{eq:dynamics_stochastic_setting}
\frac{\text{d}}{\text{d}t}
x(t,\xi)
= f(x(t,\xi),u(t,\xi),\xi),
\quad
x(0,\xi)=x_0, \text{ for every } \xi \in \Xi.
\end{equation}
\end{problem}

We will focus in the article, in particular in Section \ref{sec:convergence_analysis}, on a more general formulation of the problem, given by Problem~\ref{pb:problem_optimal_control}.
The problem is formulated on the Hilbert space $\mathcal{H}:= L^2(0,T;\R^m)$.

\begin{problem}
\label{pb:problem_optimal_control}
Given an initial condition $x_0 \in \R^n$, a running cost $\ell \colon \R^m \times \R^n  \rightarrow \R$, a final cost $g \colon \R^n \rightarrow \R$, additional costs $\psi_j \colon \R^n \rightarrow \bar{\R}$ for $j=0,\ldots,J$, we consider
\begin{equation}
\label{eq:problem}
\inf_{u \in \mathcal{H}}
\phi(u) := \phi_1(u) + \phi_2(u),
\quad \text{where: }
\begin{cases}
\begin{array}{l}
{\displaystyle \phi_1(u)=
\int_0^T \ell(u(t),x[u](t)) \ddt
+ g \left(x[u](T)\right), }\\[1em]
{\displaystyle \phi_2(u)=
\sum_{j=0}^{J-1} \int_{\tau_j}^{\tau_{j+1}} \psi_j\left(u(t)\right) \ddt, }
\end{array}
\end{cases}
\end{equation}
and where $x[u]$ denotes the solution to
\begin{equation}
\label{eq:dynamics_intro}
\dot{x}(t)=f(u(t),x(t)), \quad
x(0)= x_0.
\end{equation}
\end{problem}
It is relatively straightforward to recast Problem \ref{pb:problem_for_pha} as a particular instance of Problem \ref{pb:problem_optimal_control}: one needs to see the stochastic controls in $L^2((0,T)\times \Xi; \R^M)$ as deterministic functions in $L^2((0,T); \R^m)$ with $m= M S$, subject to the control constraint $u(t) \in U^S$. Both the control constraint and the non-anticipativity constraint can be encoded in the functions $\psi_j$, which can be taken as indicator functions (see Section \ref{sec:example} for further details).
In other words, in the context of the stochastic optimal control problem, the function $\phi_1$ coincides with the expectation term while $\phi_2$ coincides with the constraint $u \in \mathcal{U} \cap \mathcal{N}$. Then $\phi_1$ is separable with respect to the scenarios, while $\phi_2$ is separable with respect to time, which makes the computation of their proximal operators tractable.
The algorithms for solving stochastic control problems, discussed in the next paragraph, rely on such a decomposition.

\paragraph*{Scenario-based optimization algorithms}

A class of computational techniques known as progressive hedging methods has been developed to address multistage stochastic optimization problems like Problem~\ref{pb:problem_for_pha}. Rockafellar and Wets originally formulated these methods
\cite{rockafellar1976monotone, rockafellar2018solving, rockafellar2009nonanticipativity, rockafellar1991scenarios, rockafellar2017stochastic} for solving multistage convex stochastic optimization problems, and were later extended by Rockafellar and Sun \cite{rockafellar2019solving, rockafellar2020solvinglagrangian} to accommodate monotone stochastic variational inequalities. The most representative of these methods is the Progressive Hedging Algorithm (PHA), which was first introduced in \cite{rockafellar1991scenarios}.

Progressive hedging methods address multistage stochastic problems by decomposing them across scenarios, enabling independent resolution of smaller subproblems associated with each scenario. These subproblems are typically amenable to efficient solutions using standard off-the-shelf optimization software.

Progressive hedging methods have been little studied in the context of nonlinear optimal control problems like Problem~\ref{pb:problem_for_pha}, due to their lack of convexity.
In \cite{sun2024risk}, the authors studied an optimal control model incorporating risk measures and proposed a modified version of the PHA. While their approach extends progressive hedging to handle both risk measures and nonconvex optimal control problems, the resulting convergence assumptions are quite strong, requiring uniform boundedness and continuity of the iterates. Such conditions are not satisfied \textit{a priori} in our framework.
In \cite{malisani2026robust}, the PHA is applied to an optimal control problem with linear dynamics in the control variable, resulting in a convex formulation. The method was shown to converge by invoking existing convergence results for PHA in convex settings. However, our model involves nonlinear dynamics, which leads to nonconvex problems; thus, these convergence results do not directly apply.

There is also a connection between PHA and Augmented Lagrangian Methods (ALM); see, for instance, \cite{atenas2023bundle}. When applied to optimal control problems, ALM has mainly been used in settings with linear dynamics; see, e.g \cite{ito2008lagrange}. Nonlinear dynamics have been considered in \cite{aguiar2020augmented}, where the authors used ALM to solve a system of differential-algebraic equations derived from Pontryagin's maximum principle. However, their convergence analysis also relies on compactness assumptions, without sufficient regularity or coercivity conditions to ensure that these assumptions hold.

\subsection{ The Douglas-Rachford Splitting method}

\label{subsec:DRS_method_review}

The Douglas-Rachford Splitting (DRS) algorithm is designed to numerically solve composite minimization problems of the following form.

\begin{problem}
\label{pb:problem_hilbert}
Let $\phi_1,\phi_2 : H \mapsto \R \cup \{ +\infty \}$. We consider the following general problem:
\begin{equation}
\label{eq:pb_general_formulation}
\inf_{u \in H} \phi(u) := \phi_1(u) + \phi_2(u).
\end{equation}
\end{problem}
Starting from \( u \in H \), the Douglas-Rachford iteration applied to Problem~\ref{pb:problem_hilbert} with stepsize \( \gamma > 0 \) is given by
\begin{equation}
\label{eq:DRS}
\left\{
\begin{aligned}
v &\in \prox{\gamma \phi_1}(u), \\
w &\in \prox{\gamma \phi_2}(2v - u), \\
u^+ &= u + w - v.
\end{aligned}
\right.
\end{equation}
The Douglas–Rachford algorithm is particularly well-suited to cases in which $\prox{\gamma \phi_1}$ and $\prox{\gamma \phi_2}$ are easy to compute (or even have a closed form).
The Douglas-Rachford algorithm was originally proposed by Douglas and Rachford for addressing Problem~\ref{pb:problem_hilbert}, when $\phi_1,\phi_2 \in \Gamma_0(H)$ \cite{douglas1956numerical}. It was later analyzed by Lions and Mercier \cite{lions1979splitting}, also see the works of Eckstein, Bertsekas, and Combettes \cite{eckstein1992douglas}. In this convex setting, the sequence generated by the scheme converges weakly to a stationary point of $\phi$ for arbitrary stepsize $\gamma > 0$, provided a minimizer exists \cite[Thm. 7.7]{carlier2021classical}.

\paragraph*{The nonconvex case} 
Early works on nonconvex DRS focused on projection algorithms for solving feasibility problems; see, e.g., \cite{li2016douglas, bauschke2014local, bauschke2014method}. Subsequent research addressed minimization problems involving the sum of a weakly convex and a strongly convex function in \cite{dao2019adaptive, giselsson2021compositions, guo2017convergence}.

The work most closely related to our main problem is \cite{themelis2020douglas}. The authors analyzed the convergence of the DRS algorithm under the assumptions that \( \phi_1 \) is smooth with Lipschitz continuous gradient and \( \phi_2 \) is proper and lsc.

The authors in \cite{themelis2020douglas} showed that if a sequence \( (u_k, v_k, w_k)_{k \in \mathbb{N}} \) generated by the DRS algorithm converges \textit{strongly}, up to a subsequence, when the stepsize \( \gamma \) is smaller than a constant depending on the Lipschitz constant \( L_{\phi_1} \) of \( \nabla \phi_1 \), then the cluster points of \( (v_k)_{k \in \mathbb{N}} \) and \( (w_k)_{k \in \mathbb{N}} \) coincide. Moreover, if \( v^* \in H \) is a cluster point of the sequence, then it satisfies the necessary condition for local optimality:
\begin{equation}
\label{eq:necessary_condition_local_optim}
0 \in \hat{\partial} \phi(v^*).
\end{equation}
For further details, see Theorem~\ref{thm:subseq_convergence_DRS}. When \( L_{\phi_1} \) is unknown, the authors proposed an adaptive version of the algorithm that does not require prior knowledge of the Lipschitz constant of the gradient and preserves the convergence properties of the nonconvex scheme; see Subsection~\ref{subsec:drs_nonconvex} and Algorithm~\ref{algo:adaptive_DRS} for further details.

\subsection{ Contributions and main result}
\label{sec:main_result}

We aim at applying the DRS algorithm to Problem~\ref{pb:problem_optimal_control}, and due to the nonconvexity of \( \phi_1 \) (as defined in \eqref{eq:splitting_choice}), we rely on the theory developed in \cite{themelis2020douglas} for numerically solving Problem~\ref{pb:problem_hilbert} to establish the convergence of the method. However, ensuring that $\phi_1$ has a Lipschitz-continuous gradient would require too many restrictive assumptions on the dynamics~\eqref{eq:dynamics_intro} and on the cost functions of our problem. Thus, we developed a method that requires milder assumptions on \( \phi_1 \) while still ensuring convergence.

Another difficulty lies in the fact that our optimal control problem is an infinite dimensional problem. The results of \cite{themelis2020douglas} naturally extend to general Hilbert spaces. In particular, any strong cluster point of the sequence $(v_k)_{k \in \mathbb{N}}$ generated by the DRS algorithm (or its variants) can be shown to be stationary (see Theorem \ref{thm:subseq_convergence_DRS}). In finite dimensions, the existence of such a strong cluster point is relatively easy, as one only needs to show that the sequence is bounded. In our case, if the sequence is bounded, we can only deduce its weak subsequential convergence. However, we cannot directly conclude that the weak cluster points are stationary for \( \phi \) based on their proof arguments. In fact, for adapting the arguments of the convergence analysis in \cite{themelis2020douglas}, we need boundedness of the iterates in a strong compact topology of $\bfcalH$ to get a (strong subsequential) convergence of the DRS iterates.

Our article addresses these two issues. Our contributions are summarized as follows:
\begin{enumerate}
\item We propose in Section~\ref{seq:DRS_extension} an extension of the adaptive DRS algorithm of \cite{themelis2020douglas} for numerically solving Problem~\ref{pb:problem_hilbert}. This extension is Algorithm~\ref{algo:adaptive_extended_DRS} and it only requires $\nabla \phi_1$ to be locally Lipschitz continuous.
\item We establish the existence of cluster points for the strong topology of $\bfcalH$ for any sequence generated by Algorithm~\ref{algo:adaptive_extended_DRS} applied to Problem~\ref{pb:problem_optimal_control}. 
These cluster points are necessarily stationary points of the problem. This leads us to our main convergence results, Theorem \ref{thm:main_result}.
\end{enumerate}

\subsection{ Organization of the paper}
\label{subsec:organization_paper}

As we rely on the results of \cite{themelis2020douglas}, we review in Section~\ref{sec:background} some properties of the proximal operator in the nonconvex setting, as well as the main convergence results established in \cite{themelis2020douglas} and their adaptive DRS algorithm, presented as Algorithm~\ref{algo:adaptive_DRS}. We include to this section comments concerning the validity of their results in an infinite-dimensional setting.
We propose in Section~\ref{seq:DRS_extension} an extension of Algorithm~\ref{algo:adaptive_DRS} when $\phi_1$ has a locally Lipschitz gradient.
Then, in Section~\ref{sec:convergence_analysis}, we analyze the convergence of Algorithm~\ref{algo:adaptive_extended_DRS}, when applied to the optimal control problem~\ref{pb:problem_optimal_control}. In particular, we show the existence of strong cluster points for the sequence of iterates.
Finally, we present a numerical example in Section~\ref{sec:example} for illustrating the convergence of Algorithm~\ref{algo:adaptive_extended_DRS}.

\section{DRS in a non-convex setting}
\label{sec:background}

In this section we present the DRS algorithm from \cite{themelis2020douglas}, operating in a non-convex setting. We provide the reader with the main convergence results of \cite{themelis2020douglas}.
All along this section, we will consider the general Problem \ref{pb:problem_hilbert} aiming at minimizing $\phi= \phi_1 + \phi_2$. The functions $\phi_1$ and $\phi_2$ are defined on a Hilbert space $H$, possibly of infinite dimension, equipped with the norm $\| \cdot \|$. The results of \cite{themelis2020douglas} were established in a finite-dimensional setting. All quantitative results of \cite{themelis2020douglas} immediately extend to the case of an infinite-dimensional vector space; this is the case for example of Proposition \ref{prop:prox_regularity} and Proposition \ref{fact:property_dre} below. Only the results relying on topological arguments require special attention. This is the case of Lemmas \ref{lemm:prox_nonempty} and \ref{lemm:outer_semicontinuity} and the fourth statement in Theorem \ref{thm:subseq_convergence_DRS}. Elementary proofs are given for the two lemmas, while the adaptation of Theorem \ref{thm:subseq_convergence_DRS}-(iv) is discussed in Remark \ref{remark:inf_dim}.

We make the following assumption.

\begin{assump}
\label{assmptn:regularity_patrinons_infite_dim}
\begin{enumerate}
\item
Both functions $\phi_1$ and $\phi_2$ are lower bounded and weakly lower semicontinuous.
\item The function $\phi_1$ is continuously differentiable with Lipschitz continuous gradient.
\item The function $\phi_2$ is proper (i.e., not identically equal to $+\infty$).
\item Problem~\ref{pb:problem_hilbert} admits a solution: $\argmin \phi \neq \emptyset$.
\end{enumerate}
\end{assump}

\subsection{ Proximal operator and Douglas-Rachford envelope}
\label{sec:proximal_mapping}

We give in this subsection some useful properties of the proximal operator. Let us emphasize that unless otherwise stated, the function we work with is not assumed to be convex.
The following lemma ensures that the prox operators of the functions $\phi_1$ and $\phi_2$ are never empty, thanks to Assumption \ref{assmptn:regularity_patrinons_infite_dim}.

\begin{lemm}
\label{lemm:prox_nonempty}
Let $\varphi \colon H \rightarrow \bar{\R}$ be lower bounded, proper, and weakly lower semicontinuous. Then for all $\gamma>0$ and all $x \in H$, $\prox{\gamma \phi}(x)$ is non-empty.
\end{lemm}

\begin{proof}
From the assumptions of the lemma, we immediately infer that the function $y \in H \mapsto \varphi(y) + \frac{1}{2 \gamma} \norm{y-x}^2$ is coercive, proper, and weakly lower semicontinuous. Therefore it has a minimizing sequence, which is necessarily bounded, and therefore admits at least one weak accumulation point. Such an accumulation point is a minimizer of the function, which concludes the proof.
\end{proof}

\begin{lemm}
\label{lemm:outer_semicontinuity}
Let $\varphi \colon H \rightarrow \bar{\R}$ be proper and strongly lower semicontinuous.
Let $(x^{(k)})_{k \in \mathbb{N}}$ and $(y^{(k)})_{k \in \mathbb{N}}$ be two strongly convergent sequences in $H$, with limits $x$ and $y$. Assume that for any $k \in \mathbb{N}$, $y^{(k)} \in \prox{\gamma \varphi}(x^{(k)})$. Then
$y \in \prox{\gamma \varphi}(x)$.
\end{lemm}

\begin{proof}
See Appendix~\ref{appendix:proof_outer_semicontinuity}.
\end{proof}

\begin{lemm}
\label{lemm:monotonicity_prox}
Let $\varphi \colon H \rightarrow \bar{\R}$. Let $0<\gamma' < \gamma$ and let $x \in H$. For any $y \in \prox{\gamma \varphi}(x)$ and for any $y' \in \prox{\gamma' \varphi}(x)$, it holds that
\begin{equation*}
\| y' - x \| \leq \| y-x \|.
\end{equation*}
\end{lemm}

\begin{proof}
See Appendix~\ref{appendix:proof_monot}.
\end{proof}

The next proposition provides us with useful monotonicity properties of the prox operator under a smallness assumption on the stepsize $\gamma$.

\begin{prop}[Adapted from {\cite[Proposition~2.3]{themelis2020douglas}}]
\label{prop:prox_regularity}
Let $\varphi \in C^{1,1}(H)$ be $L_\varphi$-smooth. For all $\gamma < \nicefrac{1}{L_\varphi}$, the following holds:
\begin{enumerate}
\item $\prox{\gamma \varphi}$ is single-valued, and for all $x \in H$ it holds that $y = \prox{\gamma \varphi}(x)$ iff $x = y + \nabla \varphi (y)$.
\item \label{prop:prox_regularity.2} $\prox{\gamma \varphi}$ is $\left(\frac{1}{1 + \gamma L_\varphi}\right)$-strongly monotone and $(1-\gamma L_\varphi)$-cocoercive, in the sense that
\begin{equation*}
\sclr{u-u',x-x'} \geq \frac{1}{1+\gamma L_\varphi} \norm{x-x'}^2 \quad \text{and} \quad \sclr{u-u',x-x'} \geq (1 - \gamma L_\varphi) \norm{x-x'}^2,
\end{equation*}
for all $x,x' \in H$, where $u=\prox{\gamma \varphi}(x)$ and $u' = \prox{\gamma \varphi}(x')$. In particular
\begin{equation*}
\frac{1}{1+\gamma L_\varphi} \norm{x-x'} \leq \norm{u - u'} \leq \frac{1}{1 - \gamma L_\varphi} \norm{x - x'}.
\end{equation*}
Thus, $\prox{\gamma \varphi}$ is a $\frac{1}{1+\gamma L_\varphi}$-Lipschitz and invertible mapping, and its inverse $\text{Id} + \gamma \nabla \varphi$ is $(1+\gamma L_\varphi)$-Lipschitz continuous.
\end{enumerate}
\end{prop}

The convergence analysis of the DRS algorithm and its adaptive variant in \cite{themelis2020douglas} relies on the so‑called Douglas–Rachford envelope (introduced in \cite{patrinos2014douglas}), an augmented Lagrangian function that serves as a Lyapunov potential for the DRS scheme.
Consider the equivalent formulation of Problem~\ref{pb:problem_hilbert}:
\begin{equation*}
\min_{x_1, \, x_2 \in H} \phi_1(x_1)+\phi_2(x_2) \quad  \text{ s.t.} \ x_1 - x_2 = 0.
\end{equation*}
The associated augmented Lagrangian with penalty parameter $\beta>0$ is defined as
\begin{equation}
\label{eq:augmented_lagrangian}
\mathscr{L}_{\beta}(x_1,x_2,y) := \phi_1(x_1) +\phi_2(x_2)+\sclr{y,x_1-x_2}+\frac{\beta}{2} \norm{x_1-x_2}^2.
\end{equation}

\begin{definition}
Let $u \in H$, $\gamma >0$, and $v \in \prox{\gamma \phi_1}(u)$. We denote as $\phi_\gamma^{\mathrm{DR}}(u)$ the Douglas-Rachford envelope, defined as
\begin{equation*}
\phi_\gamma^{\mathrm{DR}}(u)
= \inf_{w \in H}
\left\{
\phi_2(w) +
\phi_1(v) + \langle \nabla \phi_1(v),w-v \rangle
+ \frac{1}{2\gamma} \| w-v \|^2
\right\}.
\end{equation*}
Given $w \in \prox{\gamma \phi_2}(2v-u)$, it holds that
\begin{equation}
\label{eq:augmented_lagrangian_connection}
\phi_\gamma^\mathrm{DR}(u) = \mathscr{L}_{\nicefrac1\gamma}(v,w,\gamma^{-1}(v-u)),
\end{equation}
which in particular shows that $\phi_\gamma^{\mathrm{DR}}(u)$ is real-valued.
\end{definition}

\begin{rema}
The notation $\phi_\gamma^{\mathrm{DR}}(u)$ is slightly ambiguous, in so far as it hides the dependence of the definition with respect to the choice of $v \in \prox{\gamma \phi_1}(u)$. However when $\gamma < \nicefrac{1}{L_{\phi_1}}$, we know by Proposition \ref{prop:prox_regularity} that $v$ is uniquely defined.
\end{rema}

\subsection{ Algorithm description and convergence results}
\label{subsec:drs_nonconvex}

\subsubsection{ Constant-stepsize algorithm}

We first consider the case of the DRS algorithm with constant stepsize $\gamma \in (0,\nicefrac{1}{2 L_{\phi_1}})$. It is described in 
Algorithm \ref{algo:simpleDRS} and relies on the basic DRS update defined in Algorithm \ref{algo:DRS_update}. Note that from time to time, we may call the output points $v$, $w$, and $u^+$ of Algorithm \ref{algo:simpleDRS} the successors of $u$.

\begin{algorithm}
\caption{
\parbox[t]{\linewidth}{
DRS update \label{algo:DRS_update}}
}
\begin{algorithmic}[1]
\setlength\baselineskip{1.25\baselineskip}
\Require
$u \in H$, $\gamma > 0$.
\State $v \gets \prox{\gamma \phi_1}(u)$
\State $w \gets \prox{\gamma \phi_2}(2v-u)$
\State $u^+ \gets u + w-v$
\State $\mathscr{L} \gets \mathscr{L}_{1/\gamma}(v,w,\gamma^{-1}(v-u))$ \\
\Return $\DRS(u;\gamma):=(\mathcal{L},v,w,u^+)$
\end{algorithmic}
\end{algorithm}

\begin{algorithm}
\caption{
\parbox[t]{\linewidth}{
DRS algorithm with constant stepsize 
\label{algo:simpleDRS} \\
$\text{DRS}(\cdot \, ; \cdot )$ is defined in Algorithm \ref{algo:DRS_update}
}
}
\begin{algorithmic}[1]
\setlength\baselineskip{1.25\baselineskip}
\Require
\(u_{\text{init}}\in H\) and $0 < \gamma < \nicefrac{1}{2 L_{\phi_1}}.$
\item[{\bf Initialize:}]
$u^{(0)}=u_{\text{init}}$,
\For{\(k = 0,1,2,\ldots\)}
\State $(\mathscr{L}^{(k)}, v^{(k)},w^{(k)},u^{(k+1)})
\gets \DRS(u^{(k)};\gamma)$;
\EndFor
\end{algorithmic}
\end{algorithm}

The Douglas-Rachford envelope plays a crucial role in the convergence analysis of Algorithm \ref{algo:simpleDRS}. It relies on two crucial properties, described in the next proposition: if $\gamma$ is sufficiently small, then (i) minimizing $\phi$ is equivalent to minimizing $\phi_{\gamma}^{\mathrm{DR}}$, (ii) one iteration of the DRS algorithm generates a quantifiable decay of $\phi_{\gamma}^{\mathrm{DR}}$.

\begin{prop}[Adapted from {\cite[Prop.~3.2, Prop.~3.3, Thm.~3.4, Thm.~4.1, and Rem.~4.2]{themelis2020douglas}}]
\label{fact:property_dre}
Let $0 < \gamma < \nicefrac{1}{L_{\phi_1}}$. Then the following holds true.
\begin{enumerate}
\item The Douglas-Rachford envelope $\phi_\gamma^{\mathrm{DR}}$ is continuous.
\item Minimization equivalence:
\begin{equation*}
\inf \phi = \inf \phi_{\gamma}^{\mathrm{DR}}
\quad \text{and} \quad
\argmin \phi = \prox{\gamma \phi_1}(\argmin \phi_{\gamma}^{\mathrm{DR}}).
\end{equation*}
\item Let $u \in H$. Let $v \in H$, $w \in H$, and $u^+$ be as in Algorithm \ref{algo:DRS_update}. Then
\begin{equation*}
\phi_{\gamma}^{\mathrm{DR}}(u) \leq \phi(v) \quad \text{and} \quad
\phi(w) \leq \phi_{\gamma}^{\mathrm{DR}}(u) - \frac{1 - \gamma L_{\phi_1}}{2\gamma} \|v-w\|^2.
\end{equation*}
\item If $\gamma < \nicefrac{1}{2L_{\phi_1}}$, then
\begin{equation}
\label{eq:suff_decrease_condition}
\phi_\gamma^\mathrm{DR}(u) - \phi_\gamma^\mathrm{DR}(u^+) \geq \frac{c}{(1+\gamma L_{\phi_1})^2} \norm{v-w}^2,
\quad \text{where: } c=\frac{1}{2 \gamma} - L_{\phi_1}.
\end{equation}
\end{enumerate}
\end{prop}

The convergence result established in the next theorem is based on an estimate of the decay the Douglas-Rachford envelope, obtained by summing up inequality \eqref{eq:suff_decrease_condition} at each iteration of Algorithm \ref{algo:simpleDRS}. This yields:
\begin{equation*}
\sum_{k=0}^K \| v^{(k)}-w^{(k)} \|^2
\leq \frac{(1 + \gamma L_{\phi_1})^2}{c} \Big( \phi_\gamma^{\mathrm{DR}}(u_{\text{init}}) - \inf \phi_{\gamma}^{\mathrm{DR}} \Big).
\end{equation*}

\begin{theo}[Subsequential convergence of DRS — adapted from {\cite[Thm.~4.3]{themelis2020douglas}}]
\label{thm:subseq_convergence_DRS}
Suppose that Assumption~\ref{assmptn:regularity_patrinons_infite_dim} is satisfied. Let $(\mathscr{L}^{(k)},u^{(k)},v^{(k)},w^{(k)})_{k \in \mathbb{N}}$ denote the sequence generated by Algorithm \ref{algo:simpleDRS}, executed with $\gamma \in (0,\nicefrac{1}{2L_{\phi_1}})$.
The following hold:
\begin{enumerate}
\item The residual $(v^{(k)} - w^{(k)})_{k \in \mathbb{N}}$ strongly converges to 0. Moreover, 
\[
\mathrm{\mathbf{min}}_{i \leq k} \norm{v^{(i)} - w^{(i)}} = o\left(\frac{1}{\sqrt{k}}\right).
\]
\item For each $k \in \mathbb{N}$, one has $\mathrm{dist}\!\left(0,\hat{\partial}\phi(w^{(k)})\right)
\le \frac{1+\gamma L_{\phi_1}}{\gamma}\,\| v^{(k)} - w^{(k)}\|.$
\item The sequences $(v^{(k)})_{k \in \mathbb{N}}$ and $(w^{(k)})_{k \in \mathbb{N}}$ admit the same sets of strong cluster points.
\item Any strong cluster point $x$ of the sequences $(v^{(k)})_{k \in \mathbb{N}}$ and $(w^{(k)})_{k \in \mathbb{N}}$ is a stationary point of $\phi$, i.e., \( 0 \in \hat{\partial}\phi(w^*) \).
Moreover,
\begin{equation*}
\phi(x)= \phi_{\gamma}^{\mathrm{DR}}(x)
= \lim_{k \to \infty} \phi_{\gamma}^{\mathrm{DR}}(u^{(k)}).
\end{equation*}
\end{enumerate}
\end{theo}

\begin{rema}
\label{remark:inf_dim}
\begin{enumerate}
\item The proof of the fourth statement in the theorem, provided in \cite{themelis2020douglas}, relies on the preparatory results that we have reproduced above. It also relies on the outer-semicontinuity of the prox-operator. This property corresponds to the statement of Lemma \ref{lemm:outer_semicontinuity}, which we have demonstrated for completeness.
\item Theorem \ref{thm:subseq_convergence_DRS} does not say anything concerning the existence of a strong cluster point for the sequences $(v^{(k)})_{k \in \mathbb{N}}$ and $(w^{(k)})_{k \in \mathbb{N}}$. In a finite-dimensional vector space, it suffices to show that one of these sequences is bounded, which is the case if $\phi$ has bounded level-sets (see \cite[Thm.\@ 4.3(iii)]{themelis2020douglas}). The boundedness of the sequences is not sufficient when $H$ has infinite dimension. A tailored analysis will be necessary to the existence of a strong cluster point in the context of optimal control problems.
\item It is not clear whether weak cluster points of the sequences $(v^{(k)})_{k \in \mathbb{N}}$ and $(w^{(k)})_{k \in \mathbb{N}}$ are stationary. A direct adaptation of the proof technique of \cite{themelis2020douglas} is not possible because the result of Lemma \ref{lemm:outer_semicontinuity} does not hold if the sequences $(x^{(k)})_{k \in \mathbb{N}}$ and $(y^{(k)})_{k \in \mathbb{N}}$ only weakly converges.
\end{enumerate}
\end{rema}

\subsubsection{ Adaptive variant in the non-convex case} \label{para:adaptive_variant_nonconvex_case}

The authors of \cite{themelis2020douglas} proposed an adaptive variant of the DRS algorithm for situations in which the Lipschitz constant \(L_{\phi_1}\) is not readily available. In such a case, Algorithm \ref{algo:simpleDRS} is not directly applicable as its convergence is only guaranteed for $\gamma < \nicefrac{1}{2L_{\phi_1}}$. However the sole purpose of this condition is to ensure that the Douglas-Rachford envelope satisfies the following decay condition:
\begin{equation}
\label{eq:decay_condition}
\phi_\gamma^\mathrm{DR}(u) - \phi_\gamma^\mathrm{DR}(u^+) \geq \frac{c}{(1+\gamma L)^2} \norm{v-w}^2,
\end{equation}
for positive values of $L$ and $c$.
This condition is enforced via an adaptive mechanism involving the stepsize $\gamma$ and two parameters $L$ and $c$, updated as follows.
If the decay condition is not satisfied for the current iterate at the current values of $\gamma$, $L$, and $c$, then $\gamma$ is halved, $L$ is doubled, and $c$ is set to $c = \frac{1}{2\gamma} - L$.
Once $\gamma < \nicefrac{1}{2L_{\phi_1}}$ and $L \geq L_{\phi_1}$, 
the decay condition~\eqref{eq:decay_condition} is always satisfied 
by Proposition~\ref{fact:property_dre}, and the parameters are no longer updated. As a consequence, there are atmost finitely many updates of these parameters. Considering the iterations at which $\gamma$, $c$, and $L$ no longer change, the adaptive algorithm behaves as the constant-stepsize algorithm and convergence can be established with exactly the same arguments as  in Theorem \ref{thm:subseq_convergence_DRS}. 

The adaptive mechanism is fully described in Algorithm \ref{algo:adaptive_DRS}. At the beginning of iteration $k$, the current iterate is denoted $u_{-1}$.
The points $v_{-1}$, $w_{-1}$, and $u_0$ are the successors of $u_{-1}$, and $\mathscr{L}_{-1}$ is the value of the Douglas-Rachford envelope. At line 4, the value of the DR-envelope is evaluated at $u_0$. If the decay condition is not satisfied, then the parameters $(\gamma,c,L)$ are updated (line 6) and the successors of $u_{-1}$ (as well as the DR-envelope) corresponding to the new value of $\gamma$ are recomputed (line 7). As soon as the decay condition is satisfied, $u_{-1}$ and its successors points are stored (line 9) and updated (line 10), and the while loop is interrupted (line 11).

\begin{rema}
    Storing the sequence of points in line~9 at each iteration serves a theoretical purpose: it is only needed for the convergence analysis of both the adaptive DRS algorithm and the extended DRS algorithm presented in the 
    next section.
\end{rema}

\begin{algorithm}
\caption{
\parbox[t]{\linewidth}{
DRS algorithm with adaptive stepsize. \label{algo:adaptive_DRS}
\\
$\text{DRS}(\cdot \, ; \cdot )$ is defined in Algorithm \ref{algo:DRS_update}
}
}
\begin{algorithmic}[1]
\setlength\baselineskip{1.25\baselineskip}
\Require
\(u_{\text{init}} \in H\),~$L>0$, $\gamma>0$, so that $c:= \frac{1}{2\gamma}-L >0$
\item[{\bf Initialize:}]
$ u_{-1}=u_{\text{init}} \text{ and } (\mathcal{L}_{-1}, v_{-1},w_{-1},u_0)
=\DRS(u_{\text{init}};\gamma)
$,
\For{\(k = 0,1,2,\ldots\)}
\State \(\text{cond} \gets \text{True}\)
\While{cond}
\State
$(\mathcal{L}_0,v_0,w_0,u_1)
\gets \DRS(u_0;\gamma)$
\If{
$\mathcal L_0 {}>{} \mathcal L_{-1}
- \tfrac{c}{(1+\gamma L)^2} \|w_{-1}-v_{-1}\|^2$
{\bf or}
$\phi(v_0) >\mathcal L_0$
}
\State
$(\gamma,c,L) \gets (\nicefrac\gamma2, 2c, 2L)$
\State
$(\mathcal L_{-1},v_{-1},w_{-1},u_0)
\gets \DRS(u_{-1};\gamma)$
\Else
\State \( \left(\mathcal{L}^{(k)}, u^{(k)},v^{(k)},w^{(k)} \right) \gets \left( \mathcal{L}_{-1}, u_{-1},v_{-1},w_{-1} \right) \)
\State $(\mathcal{L}_{-1}, u_{-1},v_{-1},w_{-1}, u_0) \gets (\mathcal{L}_0,u_{0},v_{0},w_{0},u_1)$
\State \(\text{cond} \gets \text{False}\)
\EndIf
\EndWhile
\EndFor
\end{algorithmic}
\end{algorithm}

\begin{theo}[Subsequential convergence of DRS with adaptive stepsize - adapted from {\cite[Thm. 4.5]{themelis2020douglas}}]
\label{thm:subseq_convergence_adaptive_DRS}
Suppose that Assumption~\ref{assmptn:regularity_patrinons_infite_dim} is satisfied.
Let $(\mathscr{L}^{(k)},u^{(k)},v^{(k)},w^{(k)})_{k \in \mathbb{N}}$ denote the sequence generated by Algorithm \ref{algo:adaptive_DRS}.
After finitely many iterations, $\gamma$ is not updated anymore.
Moreover, the following holds:
\begin{enumerate}
\item The residual $(v^{(k)} - w^{(k)})_{k \in \mathbb{N}}$ strongly converges to 0. Moreover, 
\[
\mathrm{\mathbf{min}}_{i \leq k} \norm{v^{(i)} - w^{(i)}} = o\left(\frac{1}{\sqrt{k}}\right).
\]
\item The sequences $(v^{(k)})_{k \in \mathbb{N}}$ and $(w^{(k)})_{k \in \mathbb{N}}$ admit the same sets of strong cluster points.
\item Any strong cluster point $x$ of the sequences $(v^{(k)})_{k \in \mathbb{N}}$ and $(w^{(k)})_{k \in \mathbb{N}}$ is a stationary point of $\phi$, i.e., \( 0 \in \hat{\partial}\phi(w^*) \).
Moreover,
\begin{equation*}
\phi(x)= \phi_{\gamma}^{\mathrm{DR}}(x)
= \lim_{k \to \infty} \phi_{\gamma}^{\mathrm{DR}}(u^{(k)}).
\end{equation*}
\end{enumerate}
\end{theo}

\section{Extended DRS algorithm}
\label{seq:DRS_extension}

The objective of this section is to introduce a variant of the DRS algorithm with adaptive stepsize (Algorithm~\ref{algo:adaptive_DRS}) which can handle the situation where $\nabla \phi_1$ is locally Lipschitz continuous, while so far it was assumed in Assumption~\ref{assmptn:regularity_patrinons_infite_dim} to be globally Lipschitz continuous.
All along the section, we consider the following assumption.

\begin{assump}
\label{ssmptn:relaxed_assumption}
\begin{enumerate}
\item
Both functions $\phi_1$ and $\phi_2$ are lower bounded and weakly lower semicontinuous.
\item The function $\phi_1$ is continuously differentiable with locally Lipschitz continuous gradient.
\item The function $\phi_2$ is proper and has bounded domain.
\end{enumerate}
\end{assump}

\begin{rema}
\label{rem:existence_sol}
While we have relaxed our assumption on $\nabla \phi_1$, we now require that $\phi_2$ has bounded domain. Under these new assumptions, $\phi$ is proper, has bounded domain and is weakly lower semicontinuous, which implies the existence of minimizer.
\end{rema}

\subsection{ Algorithm}
\label{subsec:extended_drs_algorithm}

Our extension of the DRS algorithm is motivated by the following two ideas:
\begin{itemize}
\item For sufficiently small stepsizes $\gamma$, one can uniformly bound the distance $\| u^{(k)}- v^{(k)} \|$: this allows to bound the whole sequence $(u^{(k)},v^{(k)},w^{(k)})_{k \in \mathbb{N}}$.
\item Via an additional adaptive mechanism, the iterates remain in a bounded 
subset of $H$ and can be shown to coincide with a sequence obtained by applying 
the DRS algorithm to $\tilde{\phi}_1 + \phi_2$, where $\tilde{\phi}_1$ is a 
mollified version of $\phi_1$ with globally Lipschitz continuous gradient on $H$. 
This allows us to leverage Theorem~\ref{thm:subseq_convergence_adaptive_DRS}.
\end{itemize}
The first idea is formalized in the next lemma.

\begin{lemm}
\label{lemma:equiv_two_prox_operators}
Let $C>0$ and $S:=\mathrm{\mathbf{dom}}\,\phi_2 + \bar{B}(C)$. There exists a threshold $\gamma_0 > 0$ such that for all stepsizes $\gamma \in (0, \gamma_0]$, for all $u \in S$, and for all  $v \in \prox{\gamma \phi_1}(u)$,
\begin{equation}
\label{lemma:eq_7.3.1}
\norm{v-u} \leq C.
\end{equation}
In particular, for $w \in \prox{\gamma \phi_2}(2v-u)$ and $u^+= u + w-v$, we have that $u^+ \in S$.
\end{lemm}

\begin{proof}
From the definition of the prox operator, we have
\begin{equation}
\label{eq:prox_estimate}
\phi_1(v) + \frac{1}{2\gamma} \norm{v-u}^2 \leq \phi_1(u),
\end{equation}
which implies that
\begin{equation*}
\norm{v - u} \leq \sqrt{2\gamma \left( \phi_1(u) - \phi_1(v) \right)}.
\end{equation*}
Set $M= \sup_{S} \phi_1$ and define $-K$ as a lower bound of $\phi_1$. Note that $M$ is finite since $S$ is bounded and $\phi_1$ has locally Lipschitz gradient, which implies that $\phi_1$ is bounded on bounded sets. Then $\| v-u \| \leq \sqrt{2 \gamma (M+K)}$ and so inequality \eqref{lemma:eq_7.3.1} holds true whenever $\gamma \leq \gamma_0:= \frac{C^2}{2(M+K)}$.
The second part of the lemma immediately follows, since $w \in \mathrm{\mathbf{dom}}(\phi_2)$.
\end{proof}

Instead of computing explicitly the threshold $\gamma_0$ obtained in Lemma \ref{lemma:equiv_two_prox_operators}, we suggest to use an adaptive mechanism, consisting in reducing the value of $\gamma$ as long as $\| v-u \| > C$. We incorporate this mechanism to the basic DRS update. We denote it $\text{DRS}_{\mathrm{loc}}$ and describe it in Algorithm \ref{algo:drs_init}.

\medskip

\begin{rema}
We described in Algorithm \ref{algo:adaptive_DRS} a first adaptive mechanism, enforcing the decay condition. Algorithm \ref{algo:drs_init} below introduces a new adaptive mechanism, enforcing the boundedness of the iterates. We will refer to these two mechanisms as the decay and the boundedness mechanism. 
\end{rema}

\begin{algorithm}
\caption{
\parbox[t]{\linewidth}{
Localizing DRS update with distance parameter $C$
}
}
\label{algo:drs_init}
\begin{algorithmic}[1]
\setlength\baselineskip{1.25\baselineskip}
\Require
$u \in \mathrm{\mathbf{dom}}\, \phi_2, C>0, \gamma >0$, $L>0$ and $c= \frac{1}{2\gamma}-L>0$. 
\State $(\gamma',c',L') \gets (\gamma,c,L)$
\State $v \gets  \prox{\gamma' \phi_1}(u)$
\While{$\norm{v - u} > C$}
\State
$(\gamma',c',L') \gets (\nicefrac{\gamma'}2, 2c',2L')$
\State $v \gets \prox{\gamma' \phi_1}(u)$
\EndWhile
\State $w \gets \prox{\gamma' \phi_2}(2v - u)$
\State $u^+ \gets u + w-v$
\State $\mathcal{L} \gets \mathcal{L}_{1/{\gamma'}}(v,w,(v-u)/\gamma')$. \\
\Return $\DRS_{\mathrm{loc}}(u; C, \gamma,L,c)
:= \big( (\mathcal{L},v,w,u^+),(\gamma',c',L') \big)$
\end{algorithmic}
\end{algorithm}

\begin{lemm}
Algorithm \ref{algo:drs_init} terminates.
\end{lemm}

\begin{proof}
Let $\gamma_0$ be given by Lemma \ref{lemma:equiv_two_prox_operators}. If $\gamma \leq \gamma_0$, then $\| v-u\| \leq C$. So the while loop of the algorithm is executed at most finitely many times.
\end{proof}

Our extension of DRS is described in Algorithm~\ref{algo:adaptive_extended_DRS}. 
Compared to Algorithm~\ref{algo:adaptive_DRS}, the only difference (aside from 
the initialization) lies at line~3, where the successors of $u_{-1}$ are 
(re)computed using $\text{DRS}_{\text{loc}}$.

\begin{algorithm}
\caption{
\parbox[t]{\linewidth}{
Extended DRS algorithm \label{algo:adaptive_extended_DRS}
\\
$\text{DRS}_{\textrm{loc}}(\cdot \, ; \cdot, \cdot, \cdot, \cdot )$ is defined in Algorithm \ref{algo:drs_init}
}
}
\begin{algorithmic}[1]
\setlength\baselineskip{1.25\baselineskip}
\Require
\(u_{\text{init}}\in \mathrm{\mathbf{dom}}\, \phi_2\), $C>0,L>0$ and $\gamma>0$, so that $c:= \frac{1}{2\gamma}-L >0$
\item[{\bf Initialize:}]
$u_{-1}=u_{\text{init}} \text{ and }\big( (\mathcal{L}_{-1}, v_{-1},w_{-1},u_0),(\gamma,c,L)\big)
=\DRS_{\mathrm{loc}}(u_{-1};C,\gamma,L,c)
$,
\For{\(k = 0,1,2,\ldots\)}
\State \(\text{cond} \gets \text{True}\)
\State $\big( (\mathcal{L}_{-1}, v_{-1},w_{-1},u_0),(\gamma,c,L)\big)
\gets \DRS_{\mathrm{loc}}(u_{-1};C,\gamma,L,c)$
\While{cond}
\State
$\big( (\mathcal{L}_0,v_0,w_0,u_1),(\gamma,c,L)\big)
= \DRS(u_0;\gamma)$
\If{
$\mathcal L_0 {}>{} \mathcal L_{-1}
- \tfrac{c}{(1+\gamma L)^2} \|w_{-1}-v_{-1}\|^2$
{\bf or}
$\phi(v_0) >\mathcal L_0$
}
\State
$(\gamma,c,L) \gets (\nicefrac\gamma2, 2c, 2L)$
\State
$(\mathcal L_{-1},v_{-1},w_{-1},u_0)
\gets \DRS(u_{-1};\gamma)$
\Else
\State \( \left(\mathcal{L}^{(k)}, u^{(k)},v^{(k)},w^{(k)} \right) \gets \left( \mathcal{L}_{-1}, u_{-1},v_{-1},w_{-1} \right) \)
\State $(\mathcal{L}_{-1}, u_{-1},v_{-1},w_{-1}, u_0) \gets (\mathcal{L}_0,u_{0},v_{0},w_{0},u_1)$
\State \(\text{cond} \gets \text{False}\)
\EndIf
\EndWhile
\EndFor
\end{algorithmic}
\end{algorithm}

\subsection{ Convergence of the extended DRS}
\label{subsec:cv_extended_drs}

We investigate in this subsection the convergence of Algorithm \ref{algo:adaptive_extended_DRS} and its well-posedness. We will need the following lemma.

\begin{lemm}
\label{lemm:phi_1}
For any $\gamma_1>0$ and for any $C>0$, there exists a continuously differentiable function $\tilde{\phi}_1 \colon H \rightarrow \R$, with globally Lipschitz gradient, such that for any $u \in S$, with $S:=\mathrm{\mathbf{dom}}\,\phi_2 + \bar{B}(C)$, for any $\gamma \in (0,\gamma_1]$, and for any $v \in \prox{\gamma \phi_1}(u)$
\begin{equation*}
    \tilde{\phi}_1(v)= \phi_1(v),
    \quad
    v \in \prox{\gamma \tilde{\phi}_1}(u),
    \quad \text{and} \quad
    \tilde{\phi}_\gamma^{\mathrm{DR}}(u)
    = \phi_\gamma^{\mathrm{DR}}(u),
\end{equation*}
where $\tilde{\phi}_\gamma^{\mathrm{DR}}$ denotes the Douglas-Rachford envelope associated with $\tilde{\phi}_1 + \phi_2$.
\end{lemm}

\begin{proof}
Let $\psi \in C^{\infty}(\mathbb{R}^+)$  be such that
\begin{equation}
\label{eq:properties_psi}
\psi(x) =
\begin{cases}
1 & \text{if } x \in [0,1], \\
\in [0,1] & \text{if } x \in [1,2], \\
0 & \text{if } x \geq 2.
\end{cases}
\end{equation}
Such a function exists, an example is given by
$\psi(x)= \exp\left( - \frac{(x - 1)^2}{2-x} \right)$, for $x \in (1,2)$.
For any $R>0$, we define the function $\tilde{\phi}_1^R \colon H \rightarrow \R$ as
\begin{equation*}
\tilde{\phi}_1^R(x) := \phi_1(x) \psi\left(\frac{\norm{x}^2}{R^2}\right).
\end{equation*}
We will make use of the following properties, whose verification is elementary and left to the reader.
\begin{itemize}
\item $\tilde{\phi}_1^R \equiv \phi_1$ on $\bar{B}(R)$ and $\tilde{\phi}_1^R(x)= 0$ whenever $\| x \| \geq \sqrt{2} R$.
\item The function $\tilde{\phi}_1^R$ is continuously differentiable, with globally Lipschitz continuous gradient.
\item It holds that $\inf_{H} \tilde{\phi}_1^R \geq \min \big( \inf_H \phi_1, 0 \big)$.
\end{itemize}
Let $R>0$ be large enough so that $S \subseteq \bar{B}(R)$, which in particular implies that $\phi_1= \tilde{\phi}_1^R$ on $\mathbf{dom}(\phi_2)$.
Let $u \in S$ and let $v \in \prox{\gamma \phi_1}(u)$. We need to prove that $v \in \prox{\gamma \tilde{\phi}^R_1}(u)$. To achieve this, we will possibly increase the value of $R$, in a way that is independent of the choice of $u$, $\gamma$, and $v$. From the definition of $v$, we have that
\begin{equation*}
\| v - u \|^2
\leq 2 \gamma \big( \phi_1(u)- \phi_1(v) \big)
\leq 2 \gamma_1 \left( \sup_{S} \phi_1 - \inf_{H} \phi_1 \right) =: D^2.
\end{equation*}
Recall that $S$ is bounded, so $\phi_1$ is bounded on $S$, and by assumption, $\phi_1$ is bounded from below on $H$, so $D$ is finite. As a consequence, $v \in S + \bar{B}(D)$. We increase the value of $R$ so that $S + \bar{B}(D) \subseteq \bar{B}(R)$. It follows that $\tilde{\phi}_1^{R}(v) = \phi_1(v)$.

Next, let $v' \in H$. We aim to verify that
\begin{equation}
\label{eq:to_show_prox_mollification_bis}
2\gamma \Big( \tilde{\phi}_1^R(v) - \tilde{\phi}_1^R(v') \Big) \leq 
\norm{v' - u}^2 - \norm{v - u}^2.
\end{equation}
If $v' \in \bar{B}(R)$, then $\tilde{\phi}_1^R(v') = \phi_1(v')$, and \eqref{eq:to_show_prox_mollification_bis} holds. Assume that $v' \notin \bar{B}(R)$. We have
\begin{equation*}
\norm{v' - u} \geq \norm{v'} - \norm{u} \geq R - M_0 \geq 0,
\end{equation*}
where $M_0 := \sup_{s \in S} \norm{s}$. Thus, on the one hand,
\begin{equation*}
\norm{v' - u}^2 - \norm{v - u}^2 \geq (R-M_0)^2 - D^2.
\end{equation*}
On the other hand, since $\tilde{\phi}_1^R(v)= \phi_1(v)$ and $v \in S + \bar{B}(D)$, we have that
\begin{equation*}
2\gamma \left( \tilde{\phi}_1^R(v) - \tilde{\phi}_1^R(v') \right)
\leq
2 \gamma_1 \left(
\sup_{S+ \bar{B}(D)} \phi_1
- \min\left( \inf_H \phi_1, 0 \right) \right)=: M_1.
\end{equation*}
We increase $R$, if necessary, so that $(R-M_0)^2 - D^2 \geq M_1$. Then inequality \eqref{eq:to_show_prox_mollification_bis} holds and finally, $v \in \prox{\gamma \tilde{\phi}_1^R}(u)$. We have already verified that $\tilde{\phi}_1^R(v)= \phi_1(v)$. The equality of the Douglas-Rachford envelopes immediately follows. The lemma is proved, with $\tilde{\phi}_1 = \tilde{\phi}_1^R$.
\end{proof}

The following lemma establishes the termination of the adaptive procedures at each iteration of Algorithm~\ref{algo:adaptive_extended_DRS}, ensuring the algorithm is well-posed, and shows the boundedness of its iterates.

\begin{lemm}
    \label{lemm:well_posedness_algo}
    Each iteration of Algorithm \ref{algo:adaptive_extended_DRS} terminates. Moreover, for every $k \in \mathbb{N}$, it holds that
    \begin{equation*}
    \| v^{(k)}-u^{(k)} \| \leq C
    \quad \text{and} \quad
    u^{(k)} \in S, 
    \end{equation*}
    where $C$ denotes the distance parameter required by Algorithm~\ref{algo:adaptive_extended_DRS}.
\end{lemm}

\begin{proof}
Let $\tilde{\phi}_1$ be the function constructed in Lemma \ref{lemm:phi_1} where $\gamma_1$ is taken as the initial value of $\gamma$ in Algorithm \ref{algo:adaptive_extended_DRS}.
Let $k \in \mathbb{N}$. Assume that the Algorithm has run until iteration $k$ and that the iterate $u_{-1}$ lies in $S$. Then the output $(\mathcal{L}_{-1}, v_{-1},w_{-1},u_0)$ of $\text{DRS}_{\text{loc}}$, obtained at line 3 is such that $\| v_{-1}- u_{-1} \| \leq C$. The while loop executed afterwards (line 4-11) corresponds to one iteration of the adaptive DRS algorithm (Algorithm \ref{algo:adaptive_DRS}) to the function $\tilde{\phi}_1 + \phi_2$ and therefore terminates. During this loop, the successor point $v_{-1}$, $w_{-1}$ and $u_0$ are possibly re-computed, but always for a smaller value of $\gamma$ than the one that was employed. Lemma \ref{lemm:monotonicity_prox} ensures that the inequality $\| v_{-1}- u_{-1} \| \leq C$ remains true until line 10. This implies that the point $u_0$ obtained at the end of the while loop lies in $S$. At line 11, this point replaces $u_{-1}$ and serves as a starting point at the next iteration of the algorithm.
The lemma follows by induction.
\end{proof}

The result of Theorem \ref{thm:subseq_convergence_DRS}, established for 
Algorithm \ref{algo:adaptive_DRS}, also holds true for our extended DRS algorithm.

\begin{theo}
\label{theo:cv_extendedDRS}
Suppose that Assumption~\ref{ssmptn:relaxed_assumption} is satisfied.
Let $(\mathscr{L}^{(k)},u^{(k)},v^{(k)},w^{(k)})_{k \in \mathbb{N}}$ denote the sequence generated by Algorithm \ref{algo:adaptive_extended_DRS}.
After finitely many iterations, $\gamma$ is not updated anymore.
Moreover, the following holds:
\begin{enumerate}
\item The residual $(v^{(k)} - w^{(k)})_{k \in \mathbb{N}}$ strongly converges to 0. Moreover, 
\[
\mathrm{\mathbf{min}}_{i \leq k} \norm{v^{(i)} - w^{(i)}} = o\left(\frac{1}{\sqrt{k}}\right).
\]
\item The sequences $(v^{(k)})_{k \in \mathbb{N}}$ and $(w^{(k)})_{k \in \mathbb{N}}$ admit the same sets of strong cluster points.
\item Any strong cluster point $x$ of the sequences $(v^{(k)})_{k \in \mathbb{N}}$ and $(w^{(k)})_{k \in \mathbb{N}}$ is a stationary point of $\phi$, i.e., \( 0 \in \hat{\partial}\phi(w^*) \).
Moreover,
\begin{equation*}
\phi(x)= \phi_{\gamma}^{\mathrm{DR}}(x)
= \lim_{k \to \infty} \phi_{\gamma}^{\mathrm{DR}}(u^{(k)}).
\end{equation*}
\end{enumerate}
\end{theo}

\begin{proof}
Let $\gamma_0$ denote the constant obtained in Lemma \ref{lemma:equiv_two_prox_operators}.
If after some iterations of Algorithm \ref{algo:adaptive_extended_DRS} we have $\gamma \leq \gamma_0$, then the condition $\| v-u \| \leq C$ is always satisfied in Algorithm \ref{algo:drs_init}. This means that after finitely many iterations, no reduction of the stepsize is realized to localize the iterates and thus Algorithm \ref{algo:adaptive_extended_DRS} behaves exactly as the adaptive Algorithm \ref{algo:adaptive_DRS}, applied to $\tilde{\phi}= \tilde{\phi}_1 + \phi_2$. Therefore, the result of Theorem \ref{thm:subseq_convergence_DRS} holds true, for $\phi$ replaced by $\tilde{\phi}$. But any accumulation point of the sequence $w^{(k)}$ necessary lies in $\mathbf{dom}(\phi_2)$, where $\phi_1$ and $\tilde{\phi}_1$, $\nabla \phi_1$ and $\nabla \tilde{\phi}_1$, $\phi_{\gamma}^{\text{DR}}$ and $\tilde{\phi}_{\gamma}^{\text{DR}}$ coincide. So the result of the theorem also holds true for $\phi$, which concludes the proof.
\end{proof}

\begin{rema}
    \label{rema:paramter_choice_extended_drs}
    In practice, the adaptive parameter \(C\) in Algorithm~\ref{algo:adaptive_extended_DRS} should be chosen sufficiently large to avoid excessive reductions of the stepsize \(\gamma\).
    Indeed, such reductions tend to slow down the overall convergence of the method, as illustrated with the numerical example in Section~\ref{sec:example}.

    However, if \(C\) is taken too large, \(\nabla \phi_1\) may take very large values in norm on \(S\). In this situation, the decay mechanism can be activated a large number of times, particularly when the control iterates generated by the operator \(\prox{\gamma \phi_1}\) lie far from \(\mathrm{\mathbf{dom}}\,\phi_2\). This may  affect the convergence speed or even the candidate control returned by the algorithm.
\end{rema}

\section{Applying DRS to the optimal control problem}
\label{sec:convergence_analysis}

In this section, we analyze the convergence properties of Algorithm~\ref{algo:adaptive_extended_DRS} when applied to the general optimal control problem of interest, Problem~\ref{pb:problem_optimal_control}. 
Subsection~\ref{sec:ssmptn} specifies the assumptions imposed on 
Problem~\ref{pb:problem_optimal_control}. Under these assumptions, 
Theorem~\ref{theo:cv_extendedDRS} states that any strong cluster point 
of the sequence generated by Algorithm~\ref{algo:adaptive_extended_DRS}is locally optimal. 
Nevertheless, as mentioned in the previous sections, strong (subsequential) convergence of the iterates cannot be guaranteed a priori.
Subsection~\ref{subsec:intermediary_results} establishes several properties of the DRS operator, demonstrating a contraction-type behavior that preserves the time regularity of the candidate control after each iteration. We rely on these results in Subsection~\ref{sec:proof_main_result}, where they enable us to invoke Ascoli’s theorem to deduce the strong convergence of the iterates.

\subsection{ Assumptions and basic properties of the optimal control problem}
\label{sec:ssmptn}

The following assumptions are taken to hold throughout the remainder of the paper. Basic properties of Problem~\ref{pb:problem_optimal_control} that follow from these assumptions are also presented in this subsection.

\begin{assump}
\label{assump:gen_OC}
We assume the following concerning the costs $\ell$ and $g$, and the dynamics $f$:
\begin{itemize}
\item Structure: There are some functions $\ell_0 \colon \R^m \rightarrow \R$, $\ell_1 \colon \R^n \rightarrow \R$, $f_0 \colon \R^n \rightarrow \R^n$, $f_1 \colon \R^n \rightarrow \R^{n \times m}$ such that
\begin{equation*}
\ell(u,x)= \ell_0(u) + \ell_1(x)
\quad \text{and} \quad
f(u,x) = f_0(x) + f_1(x)u.
\end{equation*}
\item Regularity: All functions $\ell_0$, $\ell_1$, $g$, $f_0$, and $f_1$ are differentiable with locally Lipschitz continuous gradient. 
\item Boundedness: $\ell_1$ and $g$ are bounded from below, moreover, there exists $K>0$ such that
\begin{equation*}
\| f_0(x) \| \leq K (1 + \| x \|)
\quad \text{and} \quad
\| f_1(x) \| \leq K(1 + \| x \|),
\forall x \in \R^n.
\end{equation*}
\item Convexity: $\ell_0$ is strongly convex.
\end{itemize}
In addition, we suppose that all costs $\psi_j$, for $j \in \{ 0,\ldots,J-1\}$, are proper, convex, lsc, and have a bounded domain.
\end{assump}

We apply our extended DRS algorithm to Problem \ref{pb:problem_optimal_control} with the following decomposition,
\begin{equation}
\label{eq:splitting_choice}
\left\{
\begin{aligned}
\phi_1(u) &= \int_0^T \ell(t,u(t),x[u](t)) \ddt + g(x[u](T)), \\
\phi_2(u) &= \sum_{j=0}^{J-1} \int_{\tau_j}^{\tau_{j+1}} \psi_j(u(t)) \ddt.
\end{aligned}
\right.
\end{equation}
where $u$ is taken in the Hilbert space $\bfcalH$. Recall that $x[u]$ denotes the state associated with the control $u$, solution to~\eqref{eq:dynamics_intro}, which is well-posed by Assumption~\ref{assump:gen_OC}.

In the sequel, we will denote as $p[u]$ the costate associated with the control $u$ and the state $x[u]$. It is defined as the solution of
\begin{equation}
\label{eq:dynamics_adjoint_state}
-\dot{p}(t) = \nabla \ell_1\left( x(t) \right) + \ddD f_0\left( x(t) \right)^\top p(t) + \sum_{i=1}^n u_i(t) \ddD f_1^i\left( x(t) \right)^\top p(t), \quad
p(T) = \nabla \varphi\left( x(T) \right),
\end{equation}
where $f_1^i$ refers to the i-th column of $f_1$. 
The following lemma provides us with an elementary and useful estimate.

\begin{lemm}
\label{lemma:xp_holder}
Let $N >0$. There exists a constant $M>0$ such that for any $u \in \mathcal{H}$, if $\| u \|_2 \leq N$, then $\| x[u] \|_\infty \leq M$, $\| p[u] \|_\infty \leq M$. Moreover, both $x[u]$ and $p[u]$ are $\nicefrac{1}{2}$-H\"older continuous with modulus $M$.
\end{lemm}

\begin{proof}
See Appendix~\ref{appendix:proof_xp_holder}.    
\end{proof}

\begin{lemm}
\label{lemma:verification_assumption_loc_lip}
Under Assumption~\ref{assump:gen_OC}, the functions $\phi_1$ and $\phi_2$ defined by \eqref{eq:splitting_choice} satisfy
Assumption~\ref{ssmptn:relaxed_assumption}. In addition, $\phi_2$ is convex.
\end{lemm}

\begin{proof}
See Appendix~\ref{appendix:proof_verification_assumption_loc_lip}.   
\end{proof}

Let us briefly discuss the computation of the prox operators of $\phi_1$ and $\phi_2$. First, we have $v \in \prox{\gamma \phi_1}(u)$ if and only if $v$ solves the problem
\begin{equation}
\label{eq:problem_prox_phi1}
\inf_{y \in \bfcalH} \ \int_0^T \ell(t,y(t),x[y](t)) + \frac{1}{2\gamma} \norm{y(t) - u(t)}^2 \ddt + \varphi(x[y](T)),
\end{equation}
which is a smooth optimal control problem, numerically tractable with standard methods. Next we have that $v \in \prox{\gamma \phi_2}(u)$ if and only if
\begin{equation}
\label{eq:prox_phi2}
v(t)= \prox{\psi_j}(u(t)),
\quad
\forall j=0,\ldots,J-1,\, \text{for a.e. $t \in (\tau_j,\tau_{j+1})$}.
\end{equation}
This is an immediate consequence of the separability of $\phi_2$ with respect to time.
Since $\phi_1$ and $\phi_2$ satisfy Assumption \ref{assump:gen_OC}, the optimal control from Problem \ref{pb:problem_optimal_control} has a solution, as explained in Remark \ref{rem:existence_sol}. The next lemma is concerned with the optimality conditions.

\begin{lemm}[A necessary condition for local optimality]
\label{lemma:pmp_main_problem}
Under Assumption \ref{assump:gen_OC}, a necessary condition for $u \in \bfcalH$ to be a local minimum for Problem~\ref{pb:problem_optimal_control} is
\begin{equation}
\label{eq:characterization_stationnary_points_phi}
-H'_u\left(x[u](t), u(t), p[u](t)\right) \in \partial \psi_j \left(u(t)\right), \quad \forall j=0,\ldots,J-1, \ \text{for a.e. } t \in (\tau_j,\tau_{j+1}),
\end{equation}
where \( H : \mathbb{R}^n \times \mathbb{R}^m \times \mathbb{R}^m \to \mathbb{R} \) is defined by $H(x, u, p) := \ell_0(u) + \ell_1(x) + p^\top \left( f_0(x) + f_1(x)u \right)$, where $H'_u$ denotes the gradient of $H$ with respect to $u$.
\end{lemm}

\begin{proof}
    See Appendix~\ref{appendix:proof_pmp_main_problem}.
\end{proof}

\subsection{ Regularity results}
\label{subsec:intermediary_results}

The assumptions made in the previous subsection allow us to derive regularity results for the proximal operations related to the DRS scheme. Specifically, our choice of dynamics enable reformulation of the subproblems involved in computing $\prox{\gamma \phi_1}$ as evaluations of $\prox{\gamma \ell_0 }$ at quantities derived from Pontryagin’s maximum principle. The regularity and strongly convex assumptions made on $\ell_0$ ensure a contraction-like behavior of the DRS operator, preserving both the modulus of continuity and the uniform-in-time norm of candidate control after one iteration.

The following lemma shows that a DRS update can be expressed using the reflected operators $R_{\gamma \phi_1}$ and $R_{\gamma \phi_2}$, as defined in \eqref{eq:def_reflected}. Note that the result is classical in the context of convex functions.

\begin{lemm}
\label{lemm:drs_reflected}
Let $u \in H$, let $v \in \prox{\gamma \phi_1}(u)$, let $\tilde{v}= 2v-u$, let $w \in \prox{\gamma \phi_2}(\tilde{v})$, and let $u^+ = u + w-v$. Then,
\begin{equation*}
u^+ \in \frac{1}{2}\Big(
\mathrm{Id} + R_{\gamma \phi_2} R_{\gamma \phi_1}\Big)(u).
\end{equation*}
Concretely, this means that
\begin{equation*}
\tilde{v} \in R_{\gamma \phi_1}(u)
\quad
\text{and}
\quad
2u^+ -  u
\in R_{\gamma \phi_2}(\tilde{v}).
\end{equation*}
\end{lemm}

\begin{proof}
The fact that $\tilde{v} \in R_{\gamma \phi_1}(u)$ is an immediate consequence of the definition of the reflected operator. Then, it holds that
\begin{equation*}
2 u^+ - u
=  u + 2w -2v
= 2w- \tilde{v}
\in R_{\gamma \phi_2}(\tilde{v}),
\end{equation*}
as announced.
\end{proof}

\begin{lemm}
\label{lemma:lemma_operator_R}
Let $h \in \Gamma_0(H)$ and let $\gamma >0$. Then $R_{\gamma h}$ is $1$-Lipschitz continuous. Moreover, if $h$ is $L$-smooth and strongly convex, there exists $\theta \in (0,1)$ such that $R_{\gamma h}$ is $\theta$-Lipschitz.
\end{lemm}

For the first claim of the lemma, see \cite[Corollary 23.11(ii)]{bauschke2017convex}. For the second claim, see \cite[Theorem 1]{giselsson2016linear}.
The following lemma shows that the reflected operator $R_{\gamma \phi_1}$ preserves the piecewise H\"older regularity of the controls.

\begin{lemm}
\label{lemma:regularity_R1}
Let $u \in \bfcalH$ be piecewise $\nicefrac{1}{2}$-H\"older continuous with modulus $C_1 >0$. Assume that $\| u \|_2 \leq N$ for some constant $N >0$. Let $\gamma > 0$.
Then, there exist $\theta \in (0,1)$ and $C_2 > 0$, independent of $u$, and $C_1$, such that any $\tilde{v} \in R_{\gamma \phi_1} (u) $ is also piecewise $\nicefrac{1}{2}$-Hölder continuous with modulus $\theta C_1 + C_2$.
\end{lemm}

\begin{proof}
Let $v \in \prox{\gamma \phi_1}(u)$.
Proceeding as in the proof of Lemma \ref{lemm:phi_1}, we can construct a bound of $\| v \|_2$ that is independent of $u$ and $C_1$. The control $v$ is characterized as a solution to the optimal control problem \eqref{eq:problem_prox_phi1}. Thanks to Pontryagin's principle (see, for instance, \cite{vinter2010optimal}), we have that
\begin{equation}
\label{eq:main_a}
\nabla \ell_0 \left( v(t) \right) + \frac{1}{\gamma} \left( v(t) - u(t) \right) + f_1\left( x[v](t) \right)^\top p[v](t) = 0.
\end{equation}
Reordering, we obtain that
\begin{equation}
\left( \mathrm{Id} + \gamma \nabla \ell_0 \right) \left( v(t) \right) = u(t) + \gamma z_v(t),
\end{equation}
where
\begin{equation}
    \label{eq:def_residual}
    z_v(t):= - f_1\left( x[v](t) \right)^\top p[v](t).
\end{equation}
Therefore, as $\ell_0$ is differentiable and (strongly) convex,
\begin{equation}
\label{eq:u_rewritten}
v(t) = \prox{\gamma \ell_0} \left( u(t) + \gamma z_v (t) \right).
\end{equation}
We deduce from Lemma \ref{lemma:xp_holder} that $z_v$ is $\nicefrac{1}{2}$-Hölder continuous with modulus $\omega$, where $\omega$ only depends on $N$. It follows from \eqref{eq:u_rewritten} that
\begin{align}
\tilde{v}(t)
={} & 2v(t)-u(t) \notag \\
={} & 2 \, \prox{\gamma \ell_0} \big(u(t)+\gamma z_v (t) \big) - u(t) \notag \\
={} & R_{\gamma \ell_0} \big( u(t)+ \gamma  z_v (t) \big) + \gamma z_v (t).
\label{eq:r1_rewritten}
\end{align}
Let $\theta \in (0,1)$ be given by Lemma \ref{lemma:lemma_operator_R}.
For all $j=0,\ldots,J-1$ and for almost every $t_1$ and $t_2 \in (\tau_j,\tau_{j+1})$, we have that
\begin{equation*}
\| \tilde{v}(t_2)- \tilde{v}(t_1) \|
\leq \theta \norm{ u(t_2) -u(t_1)} + \gamma (1 + \theta) \norm{ z_v (t_2) -  z_v (t_1)}.
\end{equation*}
Using the regularity of $z_v$ and $u$, we deduce that
\begin{align*}
\norm{(R_1 u) (t_2) - (R_1 u) (t_1)} \leq {} & \theta C_1 |t_2 - t_1|^{\frac{1}{2}} + (1+\theta) \gamma \omega |t_2 - t_1|^{\frac{1}{2}} \\
\leq {} & (\theta C_1 + C_2) |t_2 - t_1|^{\frac{1}{2}},
\end{align*}
for $C_2 = (1+\theta)\gamma \omega$. This proves the lemma.
\end{proof}

To conclude this subsection, we give a regularity result concerning the operator $R_{\gamma \phi_2}$.

\begin{lemm}
\label{lemma:regularity_r2}
Let $u$ and $v \in \bfcalH$ and let $\gamma >0$. For all $j=0,\ldots,J-1$, for almost every $t_1,\, t_2 \in (\tau_j,\tau_{j+1})$, we have
\begin{equation*}
\norm{(R_{\gamma \phi_2} (u))(t_2) - (R_{\gamma \phi_2} (v))(t_1)} \leq \norm{u(t_2) - v(t_1)}.
\end{equation*}
\end{lemm}

\begin{proof}
By \eqref{eq:prox_phi2}
and Lemma \ref{lemma:lemma_operator_R}, we have
\begin{align*}
\| (R_{\gamma \phi_2}(u))(t_2)
- (R_{\gamma \phi_2}(v))(t_1) \|
= {} & \| R_{\gamma \psi_j}(u(t_2))
- R_{\gamma \psi_j}(v(t_1)) \| \\
\leq {} & \| u(t_2)-v(t_1)\|.
\end{align*}
The lemma is proved.
\end{proof}

\subsection{ Convergence of the DRS method}
\label{sec:proof_main_result}

In this subsection, we prove the existence of a strong accumulation point of the sequence $(u^{(k)},v^{(k)},w^{(k)})_{k \in \mathbb{N}}$ generated by the extended DRS algorithm, Algorithm~\ref{algo:adaptive_extended_DRS}, for solving Problem~\ref{pb:problem_optimal_control}. In view of applying Ascoli's theorem, we show that the sequence $(u^{(k)})_{k \in \mathbb{N}}$ is bounded and piecewise $\nicefrac{1}{2}$-H\"older continuous, uniformly in $k$.

From Theorem \ref{theo:cv_extendedDRS}, we know that after some iteration $k_0 \in \mathbb{N}$, the stepsize is not updated anymore. Let us denote as $\gamma$ its final value. If the initialization $u_{\text{init}}$ lies in $C_\tau^{0,\nicefrac{1}{2}}(0,T;\R^n)$, then $u^{(k_0)}$ also lies in $C_\tau^{0,\nicefrac{1}{2}}(0,T;\R^n)$, as can be established by induction, combining Lemmas \ref{lemma:lemma_operator_R}, \ref{lemma:regularity_R1}, and \ref{lemma:regularity_r2}.
In the following, we discard the first $k_0$ elements of the sequence 
$(u^{(k)}, v^{(k)}, w^{(k)})_{k \in \mathbb{N}}$ and reindex the remaining elements starting from $0$.

\begin{lemm}
\label{prop:piecewise_regularity}
The sequences $(u^{(k)})_{k \in \mathbb{N}}$ and $(v^{(k)})_{k \in \mathbb{N}}$ are uniformly bounded and piecewise $\nicefrac{1}{2}$-H\"older continuous.
\end{lemm}

\begin{proof}
\textit{Part 1: piecewise $\nicefrac{1}{2}$-H\"older continuity.}
\\
For the sake of clarity, we denote as $R_1$ and $R_2$ the mappings $R_{\gamma \phi_1}$ and $R_{\gamma \phi_2}$.
For any $k \in \mathbb{N}$, we set $\tilde{v}^{(k)}= 2v^{(k)}-u^{(k)} \in R_1(u^{(k)}$. We recall that by Lemma \ref{lemm:drs_reflected},
\begin{equation}
\label{eq:drs_update_reg}
u^{(k+1)}
= \frac{1}{2} \Big(
u^{(k)}+ R_2(\tilde{v}^{(k)})
\Big).
\end{equation}
As was explained above, $u^{(0)}$ is piecewise $\nicefrac{1}{2}$-H\"older continuous. Let $C_0$ denote its modulus.
Let $k \in \mathbb{N}$. Assume now that $u^{(k)}$ is piecewise $\nicefrac{1}{2}$-H\"older continuous with modulus $C_k$. Using Lemma~\ref{lemma:regularity_R1}, we deduce that $\tilde{v}^{(k)}$ is piecewise $\nicefrac{1}{2} $-Hölder continuous with modulus $ \theta C_k + C $, where $ \theta $ and $ C $ are independent of $k$.
It easily follows from Lemma \ref{lemma:regularity_r2} that
$R_2(\tilde{v}^{(k)})$ is also piecewise $\nicefrac{1}{2}$-H\"older continuous with modulus $\theta C_k +C$.
Finally we deduce from \eqref{eq:drs_update_reg} that $u^{(k+1)}$ is piecewise $\nicefrac{1}{2}$-H\"older continuous with modulus
$C_{k+1}= \tilde{\theta} C_k + \frac{1}{2}C$,
where $ \tilde{\theta} := \frac{1 + \theta}{2} \in (0,1) $. By a direct induction argument, we obtain that
\begin{equation}
C_{k} = C_0 \tilde{\theta}^k + \frac{C}{2} \sum_{i=0}^{k-1} \tilde{\theta}^i
\leq C_0 \tilde{\theta}^k + \frac{C}{2(1 - \tilde{\theta})} \leq C_0 + \frac{C}{2(1 - \tilde{\theta})},
\end{equation}
establishing the uniform piecewise $\nicefrac{1}{2} $-Hölder continuity of $(u^{(k)})_{k \in \mathbb{N}}$ and $(\tilde{v}^{(k)})_{k \in \mathbb{N}}$.  Since $v^{(k)}= \frac{1}{2}(\tilde{v}^{(k)}+u^{(k)})$, we also have that the sequence $(v^{(k)})_{k \in \mathbb{N}}$ is uniformly piecewise $\nicefrac{1}{2} $-Hölder continuous.
\\
\\
\textit{Part 2: uniform boundedness.}
\\
Let us set $M_k = \norm{u^{(k)}}_\infty := \sup_{t \in [0,T]} \norm{u^{(k)}(t)}$. Let us also denote by $R_0:= 2 \, \prox{\gamma \ell_0} - \mathrm{Id}$.
Then, using~\eqref{eq:r1_rewritten} and Lemma~\ref{lemma:lemma_operator_R}, we have for all $t$ that 
\begin{align*}
\norm{(R_1 (u^{(k)}))(t) - (R_1 (0_\bfcalH))(t)}
&\leq \norm{ R_0 \left( u^{(k)}(t) + \gamma z_{u^{(k)}}(t) \right)  - R_0 \left(z_{0_{\mathbf{\mathcal{H}}}}(t) \right) } \\ 
& \qquad + \gamma \norm{ z_{u^{(k)}}(t) - z_{0_{\bfcalH}}(t) } \\
&\leq \theta \norm{ u^{(k)}(t) + \gamma \left( z_{u^{(k)}}(t) - z_{0_{\bfcalH}}(t) \right) } + \gamma \norm{ { z_{u^{(k)}}(t) - z_{0_{\bfcalH}}(t)}} \\
&\leq \theta \norm{u^{(k)}(t)} + (1 + \theta)\gamma \norm{z_{u^{(k)}}(t) - { z_{0_\bfcalH}(t)}},
\end{align*}
where $z_{u^{(k)}}$ and $z_{0_{\bfcalH}}$ denote the remainder $z_v$ defined in~\eqref{eq:def_residual}, for $v = u^{(k)}$ and $v = 0_{\bfcalH}$, respectively.
Then, from Lemma~\ref{lemm:well_posedness_algo}, we know that $(u^{(k)})_{k}$ is uniformly bounded in $\bfcalH$. Using Lemma~\ref{lemma:xp_holder}, it follows that $\left(z_{u^{(k)}}\right)_{k}$ 
is uniformly bounded.
Thus,
\begin{equation*}
\norm{(R_1 (u^{(k)}))(t) - (R_1 (0_\bfcalH))(t)} \leq \theta M_k + C,
\end{equation*}
for some constant $C>0$ depending on $\gamma,\theta$ and $\sup_k \norm{u_k}_2$. Next, using Lemma~\ref{lemma:regularity_r2} and the reverse triangle inequality, we have that
\begin{align*}
\norm{(R_2 (R_1 (u^{(k)})))(t)} - \norm{(R_2 (R_1( 0_{\mathbf{\mathcal{H}}})))(t)} &\leq \norm{(R_2 (R_1 (u^{(k)})))(t) - (R_2 (R_1( 0_{\mathbf{\mathcal{H}}})))(t)} \\
&\leq \norm{(R_1 (u^{(k)}))(t) - (R_1 (0_{\mathbf{\mathcal{H}}}))(t)},
\end{align*}
so that
\begin{equation*}
\norm{(R_2(R_1(u^{(k)}) ))(t)} \leq \theta M_k + C + \norm{(R_2(R_1(0_{\mathbf{\mathcal{H}}})))(t)}.
\end{equation*}
Using the same arguments as in the first part, we conclude that the sequences $(u^{(k)})_{k \in \mathbb{N}}$, $(\tilde{v}^{(k)})_{k \in \mathbb{N}}$, and $(v^{(k)})_{k \in \mathbb{N}}$ are uniformly bounded.
\end{proof}

\begin{theo}
\label{thm:main_result}
Let Assumption \ref{assump:gen_OC} hold true. Let $(\mathscr{L}^{(k)},u^{(k)},v^{(k)},w^{(k)})_{k \in \mathbb{N}}$ denote the sequence generated by Algorithm \ref{algo:adaptive_extended_DRS}, executed with initial condition $u_{\text{init}} \in C_\tau^{0,\nicefrac{1}{2}}(0,T;\R^m)$. Then the sequences $v^{(k)}$ and $w^{(k)}$ have the same non-empty sets of strong cluster points, which are all stationary points of the optimal control problem, Problem \ref{pb:problem_optimal_control}.
\end{theo}

\begin{proof}
As the optimal control problem satisfies the abstract Assumption \ref{ssmptn:relaxed_assumption} of Section \ref{seq:DRS_extension}, we may apply Theorem \ref{theo:cv_extendedDRS}. Then it remains to show the existence of a strong accumulation for the sequence $(v^{k})_{k \in \mathbb{N}}$. Consider the restriction of $v^{(k)}$ to the first time interval $[0,\tau_1]$. Up to a modification of $v^{(k)}$ on a null-measure set, the sequence $(v^{(k)}_{|[0,\tau_1]})_{k \in \mathbb{N}}$ is uniformly bounded and equicontinuous. Thus it admits a uniformly convergent subsequence, by the Arzela-Ascoli theorem. Considering the restriction of the subsequence to the next interval, $[\tau_1,\tau_2]$, applying again the Arzela-Ascoli theorem and extracting a new subsequence, we obtain a subsequence of $(v^{(k)})_{k \in \mathbb{N}}$ whose restriction to $(0,\tau_2)$ converges in $L^\infty(0,T;\R^m)$. Repeating this procedure on the whole interval $(0,T)$, we obtain the existence of cluster point for the $L^\infty$-norm, which is then also a strong cluster point in $\mathcal{H}$.
\end{proof}

\section{Numerical example}
\label{sec:example}

\paragraph*{Link between the multistage stochastic problem and the main problem}

Problem~\ref{pb:problem_for_pha} is a particular instance of Problem~\ref{pb:problem_optimal_control}. 
Indeed, 
it suffices to set $\mathcal{H} = L^{2}(0,T;\mathbb{R}^{SM})$, identifying 
$L^{2}(0,T;\mathbb{R}^{SM})$ with $L^{2}(0,T;\mathbb{R}^{S \times M})$, so 
that each $u^s \in L^2(0,T;\mathbb{R}^M)$ is the control trajectory for 
scenario $\xi^s$. 
The constraint set $\mathcal{U}$ is the subspace of all 
$u \in \bfcalH$ such that $u^s \in U$ for all $s = 1,\ldots,S$.
Finally, with a slight abuse of notation, the 
functionals are defined by
\[
\ell(u,x) = \sum_{s=1}^S p_s\, \ell(u^s,x^s,\xi^{s}),
\quad
g(x) = \sum_{s=1}^S p_s g(x^s,\xi^{s}),
\quad
f(x,u) = \left( f(x^s,u^s,\xi^s) \right)_{s=1,\ldots,S},
\]
and for all $j \in \{1,\ldots,J-1\}$,
\[
\psi_{j}(u) = \iota_{N_{j} \cap U^S}(u),
\]
where the set $N_{j}$ is defined by
\[
N_{j}
:=
\bigg\{
u \in \mathbb{R}^{S \times M}
\ \mid \
\forall\, \xi^s, \xi^{s'} \in \Xi
\text{ indiscernible up to time } \tau_{j+1},
\
u\left(\cdot,\xi^{s}\right) = u\left(\cdot,\xi^{s'}\right)
\bigg\}. 
\]
Also, for any $u_1,u_2 \in \mathcal{U}$, $\sclr{u_1,u_2}:= \mathbb{E}_\xi \left[ \sclr{u_1(.,\xi) , u_2(.,\xi} \right]$ where $\sclr{.,.}$ denotes the usual scalar product of $L^2\left(0,T;\mathbb{R}^M\right)$; see Subsection~\ref{subsec:notations}. It is straightforward to verify that $N_{j}$ is a non-empty, closed linear subspace of
$\R^{S \times M}$. Consequently, for each $j \in \{ 0 , \ldots, J-1 \}$, we have $\psi_j \in \Gamma_0(\R^{S \times M}).$
\begin{rema}
For any $u \in \mathbb{R}^{S \times M}$, we have
$\prox{\gamma \psi_j}(u) = \proj{N_j \cap U^S}(u)
= \proj{U^S} \circ \proj{N_j}(u).$
Indeed, the second equality follows from the stability condition stated in
Lemma~\ref{lemma:stability_condition}, as $N_j$ is a closed linear space and that $U^S$ is closed convex. Therefore, since the projections onto these two spaces are typically given 
by closed-form expressions, the proximal operators of $\phi_1$ and $\phi_2$ 
are numerically tractable.
\end{rema}

\begin{lemm}
\label{lemma:stability_condition}
Let $N \subseteq H$ be a closed linear subspace and let $C \subseteq H$ be a closed convex set. In particular, the projections $\proj{N}$ and $\proj{C}$ are well defined.
Then, if the following stability condition holds:
\begin{equation}
\label{eq:stability_condition}
\proj{C} \circ \proj{N} \subseteq N,
\end{equation}
we have
\begin{equation}
\label{eq:composition_proj}
\proj{N \cap C} = \proj{C} \circ \proj{N}.
\end{equation}
\end{lemm}

\begin{proof}
See Appendix~\ref{appendix:proof_stability_condition}.    
\end{proof}

\paragraph*{Optimal control problem}

We are interested in solving a special instance of Problem~\ref{pb:problem_for_pha}. We consider the time horizon $T=10$ as well as the intermediary times $\tau_1 = 3, \tau_2 = 6$ on which a random process $\xi$ which takes values in the interval $[0,1]$ is progressively revealed. We arbitrarily choose a process which is revealed in $S=4$ scenarios through the time horizon, as described in Figure~\ref{fig:scenario_tree}.

The additional information for the problem is as follows. For every $s \in \Xi$, we assume $p_s = 0.25$. The running and terminal costs are $\ell(u,x,\xi) = \tfrac{1}{2}\,\left(u^{2} + x^{2}\right)$ and $\varphi(x,\xi) = 25\,\bigl(x - \xi(T)\bigr)^{2}$. Also, the state variable evolves according to
\[
\frac{\mathrm{d}}{\mathrm{d}t} x(t,\xi) = x(t,\xi)\,u(t,\xi),\qquad
x(0,\xi) = 0.5,
\quad \text{for all } \xi \in \Xi.
\]

One can verify that, in this framework, Assumption~\ref{assump:gen_OC} is satisfied. 
However, it can also be shown that $\nabla \phi_{1}$ is not globally Lipschitz in this setting, and therefore the adaptive DRS algorithm does not come with a guaranteed convergence result for the applied problem.

\paragraph*{Numerical test} We numerically solve the applied problem described above using Algorithm~\ref{algo:adaptive_extended_DRS}. The computations are performed with $K = 600$ iterations, an initial guess of the Lipschitz constant $L = 1$, and the corresponding initial stepsize $\gamma_{0} = 0.25$. We also test the numerical procedure with different choices of the adaptive parameter $C \in \{0.05, 0.1, 1.0\}$. We used a Differential-Algebraic Equation (DAE) solver to resolve the optimal control subproblems arising from the proximal operator applied to \(\phi_1\), as in \eqref{eq:problem_prox_phi1}. The solver is a two‑point boundary DAE method designed for index‑1 differential–algebraic systems. For further details, we refer the reader to \cite{malisani2024interior}.

The algorithm is initialized with the null control, which yields an initial cost of 4.438. The candidate state trajectories obtained after the procedure are shown in Figure~\ref{fig:state_trajectories}. The cost function decreases toward the value 0.778 for all considered values of the adaptive parameter $C$, as illustrated in Figure~\ref{fig:objective_function}. Finally, Figure~\ref{fig:residuals} shows a linear convergence rate of the fixed-point residuals, which is faster than the theoretical rates established in Theorem~\ref{theo:cv_extendedDRS}.

Table~\ref{table:stepsize_reductions} shows that the adaptive parameter $C$ should be taken sufficiently large to avoid excessive reductions of the stepsize $\gamma$. 
Indeed, as mentioned in Remark~\ref{rema:paramter_choice_extended_drs}, excessive stepsize reductions can slow down the overall convergence of the method, as illustrated in Figures~\ref{fig:objective_function} and~\ref{fig:residuals}.

\begin{table}[H]
\centering
\begin{tabular}{|c|c|c|c|}
\hline
$C$ & \# Stepsize Updates & Last Stepsize  & Updates (Mech.~1 / Mech.~2) \\
\hline
1  & 2 & $6.25\text{e}10^{-2}$ & 0 / 2 \\
\hline
0.1 & 6 & $3.90625\text{e}10^{-3}$ & 6 / 0 \\
\hline
0.05 & 7 & $1.953125\text{e}10^{-3}$ & 7 / 0 \\
\hline
\end{tabular}
\caption{\small Summary of the stepsize updates performed by Algorithm~\ref{algo:adaptive_extended_DRS} for different values of the adaptive parameter C. Iteration 0 corresponds to the Initialize step, specifically the procedure described in Algorithm~\ref{algo:drs_init}. “Mech. 1” and “Mech. 2” denote the decay and boundedness mechanisms of Algorithm~\ref{algo:adaptive_extended_DRS}, respectively. Finally, when \(C = 1\), all stepsize updates occur during the first iteration; otherwise, they take place during the initialization phase of Algorithm~\ref{algo:adaptive_extended_DRS}.}
\label{table:stepsize_reductions}
\end{table}

\begin{figure}
\centering
\begin{minipage}{0.48\linewidth}
\centering
\includegraphics[width=\linewidth]{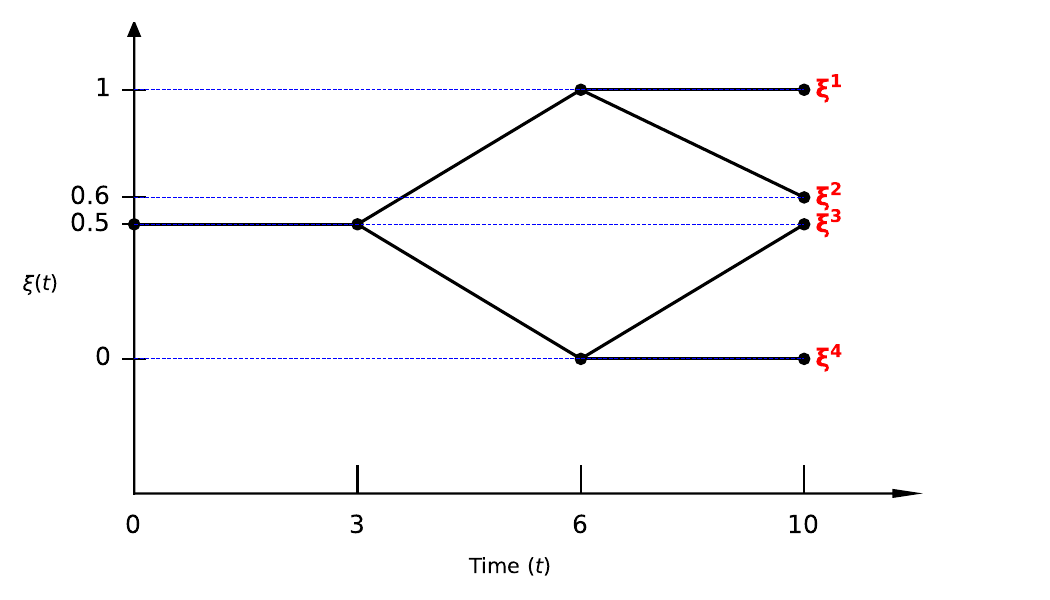}
\caption{\small Evolution of the process $\xi$ used in the numerical experiment over time. Its different trajectories form a scenario tree.}
\label{fig:scenario_tree}
\end{minipage}
\hfill
\begin{minipage}{0.48\linewidth}
\centering
\includegraphics[width=\linewidth]{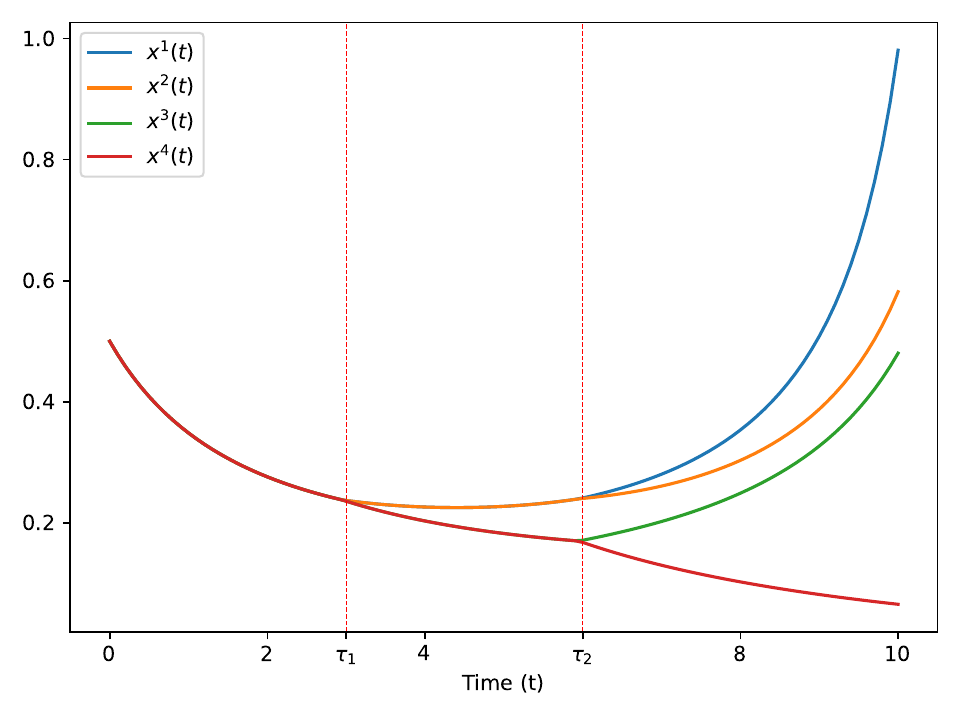}
\caption{\small State trajectories of $x := x[v^{K}]$ for each scenario after running Algorithm~\ref{algo:adaptive_extended_DRS} with the adaptive parameter $C=1$. Note the non-anticipativity of the variable $x$.}
\label{fig:state_trajectories}
\end{minipage}
\end{figure}

\begin{figure}
\centering

\begin{minipage}{0.48\linewidth}
\centering
\includegraphics[width=\linewidth]
{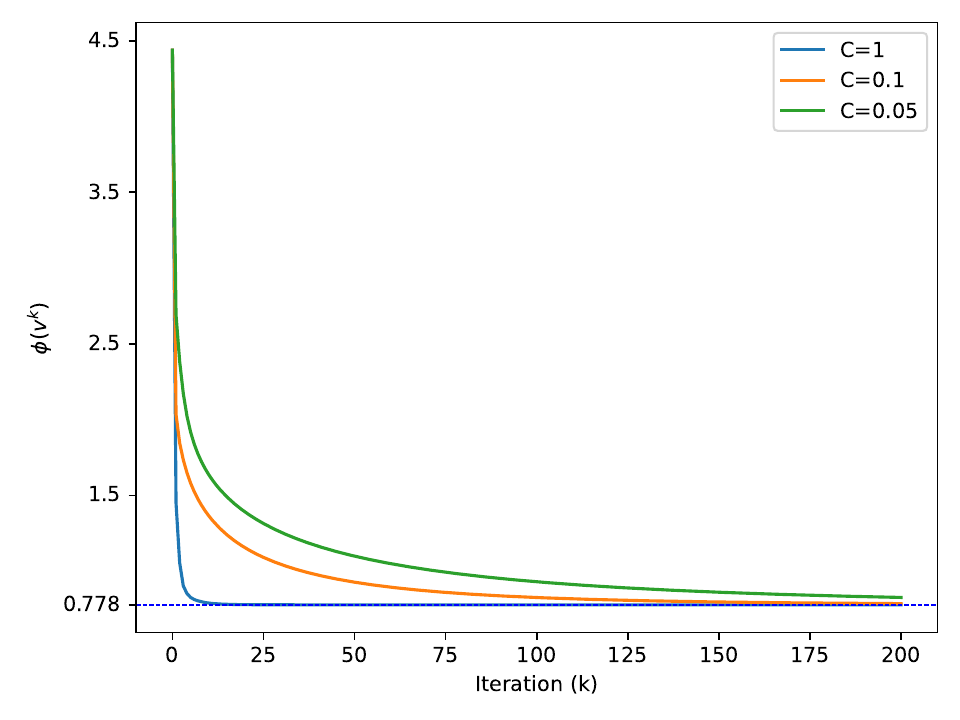}
\caption{\small Comparison of the decrease in the objective function over the first 200 iterations of Algorithm~\ref{algo:adaptive_extended_DRS} for the different values of the parameter $C$.}
\label{fig:objective_function}
\end{minipage}
\hfill
\begin{minipage}{0.48\linewidth}
\centering
\includegraphics[width=\linewidth]{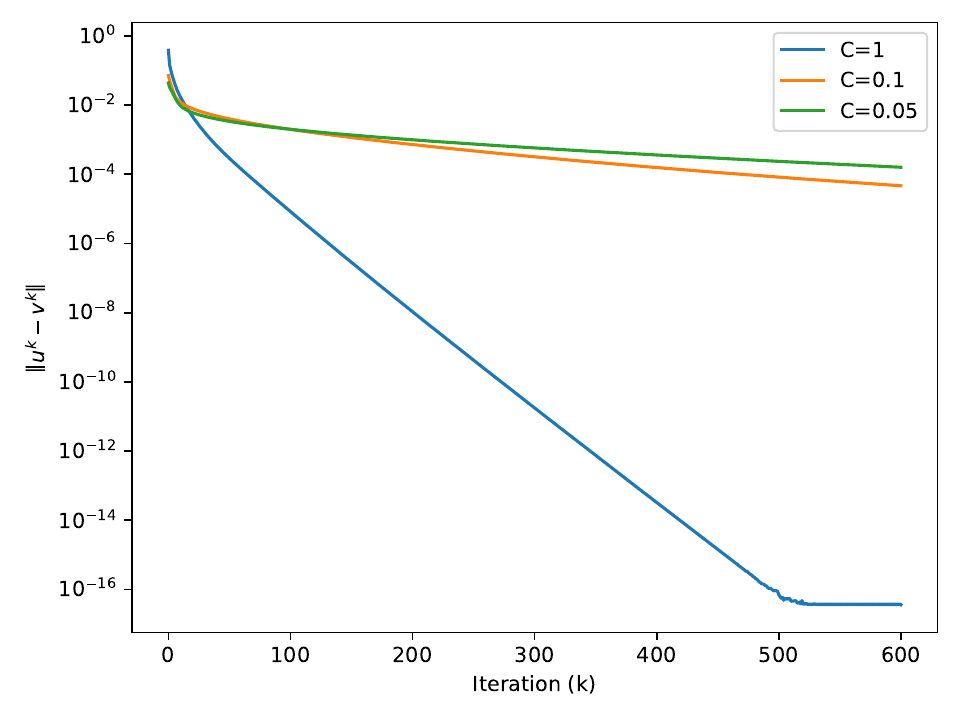}
\caption{\small Comparison of the convergence speed of the fixed-point residuals obtained when running Algorithm~\ref{algo:adaptive_extended_DRS} for different values of the parameter $C$, displayed on a semilogarithmic scale.}
\label{fig:residuals}
\end{minipage}

\end{figure}

\section{Appendix}

This section gathers the proof of some technical lemmas.

\subsection{ Proof of Lemma \ref{lemm:outer_semicontinuity}}
\label{appendix:proof_outer_semicontinuity}
For any $k \in \mathbb{N}$ and any $z \in H$, it holds that
\begin{equation}
\label{eq:lemm_outersemicont}
\varphi(y^{(k)}) + \frac{1}{2\gamma} \| y^{(k)} - x^{(k)} \|^2
\leq\varphi(z) + \frac{1}{2\gamma} \| z - x^{(k)} \|^2.
\end{equation}
Passing to the limit as $k \to \infty$, we obtain that
\begin{equation*}
\varphi(y) + \frac{1}{2\gamma} \| y - x \|^2
\leq\varphi(z) + \frac{1}{2\gamma} \| z - x \|^2.
\end{equation*}
The conclusion follows.

\subsection{ Proof of Lemma \ref{lemm:monotonicity_prox}}
\label{appendix:proof_monot}

First, we have that 
\[
\varphi(y) + \frac{1}{2\gamma}\|y-x\|^2 \;\leq\; \varphi(y') + \frac{1}{2\gamma}\|y'-x\|^2.
\]
As the same inequality holds exchanging $y$ and $y'$, we deduce that 
\[
\frac{1}{\gamma}\|y-x\|^2 + \frac{1}{\gamma'}\|y'-x\|^2
\;\leq\;
\frac{1}{\gamma}\|y'-x\|^2 + \frac{1}{\gamma'}\|y-x\|^2.
\]
Rearranging gives
\[
\Bigl(\tfrac{1}{\gamma} - \tfrac{1}{\gamma'}\Bigr)\bigl(\|y-x\|^2 - \|y'-x\|^2\bigr) \leq 0.
\]
The statement follows as $\gamma' < \gamma$. 

\subsection{ Proof of Lemma~\ref{lemma:xp_holder}}
\label{appendix:proof_xp_holder}

We first show that $x[u]$ is $\nicefrac{1}{2}$-Hölder continuous and uniformly bounded. From Assumption~\ref{assump:gen_OC},
\begin{equation*}
  x(t) = x_0 + \int_0^T f_0(x(\lambda)) + f_1(x(\lambda))\,u(\lambda)\,\mathrm{d}\lambda,
\end{equation*}
and by the same assumption,
\begin{align*}
  \norm{x(t)} &\leq \norm{x_0} + KT + K\sqrt{T}\norm{u}_2
    + K\int_0^T \norm{x(\lambda)}\bigl(1 + \norm{u(\lambda)}\bigr)\,\mathrm{d}\lambda.
\end{align*}
A direct application of Gronwall's lemma, together with the Cauchy-Schwarz inequality, yields:
\begin{equation}
\label{eq:uniform_boundedness_x}
\norm{x(t)} \leq C_{x_0,K,T,N} \exp\left(\int_0^T K\left(1+\norm{u(\lambda)}\right) \ddl \right) \leq C'_{x_0,K,T,N}.
\end{equation}
Let $C_x:= C'_{x_0,K,T,N}$. Therefore $x\in L^{\infty}(0,T;\R^m)$ with $\norm{x}_\infty$ being independent of $u$. Then, from Assumption~\ref{assump:gen_OC} again, we have
\begin{align}
\norm{\dot{x}}_{L^2}^2 &\leq \int_0^T 2 \left( \norm{f_0\left( x(\lambda) \right)}^2 + \norm{f_1\left( x(\lambda) \right) u(\lambda)}^2 \right) \ddl \\
&\leq C_{f_0,K,T,C_x} + C_{f_1,K,T,C_x} \| u \|^2_2 \\
&\leq C_{f_0,f_1,K,T,C_x}.
\label{eq:uniform_boundedness_x_dot}
\end{align}
For the second and third inequalities, we use the fact that $f_0$ and $f_1$ are bounded on compact sets containing $\|x(\lambda)\|$. It follows that $\dot{x}$ is uniformly bounded in $\bfcalH$ by a constant independent of $u$, which ensures that $x[u]$ is $\nicefrac{1}{2}$-Hölder continuous with a modulus independent of $u$.

The regularity of $p$ can be established in the same manner: an estimate of $\| p \|_\infty$ can be obtained with Gronwall's lemma. It further yields an estimate of $\| \dot{p} \|_2$, from which the $\nicefrac{1}{2}$-H\"older continuity follows.

\subsection{ Proof of Lemma \ref{lemma:verification_assumption_loc_lip}}
\label{appendix:proof_verification_assumption_loc_lip}

\noindent \textbf{\textit{Step 1: weak lower semicontinuity of $\phi_1$.}} Let $(u_k)_{k \in \mathbb{N}},\bar{u} \in \bfcalH$ such that $u_k \rightharpoonup \bar{u}$ as $k \rightarrow +\infty$.
Using the convexity of $\ell_0$, for a.e $t \in (0,T)$ we have
\begin{equation*}
\ell_0(u_k(t)) \geq \ell_0(\bar{u}(t)) + \left\langle \nabla \ell_0(\bar{u}(t)), u_k(t) - \bar{u}(t) \right\rangle.
\end{equation*}
By integrating and passing to the limit, we obtain
\begin{equation*}
\liminf_{k \to +\infty} \int_0^T \ell_0(u_k(t)) \, \ddt \geq \int_0^T \ell_0(\bar{u}(t)) \, \ddt + \liminf_{k \to +\infty} \int_0^T \left\langle \nabla \ell_0(\bar{u}(t)), u_k(t) - \bar{u}(t) \right\rangle \ddt.
\end{equation*}
The second term on the right-hand side of the inequality vanishes due to weak convergence. Therefore,
\begin{equation}
\label{eq:to_gather_1}
\liminf_{k \to +\infty} \int_0^T \ell_0(u_k(t)) \, \ddt \geq \int_0^T \ell_0(\bar{u}(t)) \, \ddt.
\end{equation}
Now, as $(u_k)_{k \in \mathbb{N}}$ is weakly convergent, is it also bounded in $\bfcalH$. It follows from Lemma~\ref{lemma:xp_holder} that $(x[u_k])_{k \in \mathbb{N}}$ is uniformly bounded and equicontinuous. Therefore, a direct application of Ascoli's theorem yields, up to a subsequence,
\begin{equation}
\label{eq:uniform_cv_xn}
x[u_k] \xrightarrow[k \to +\infty]{L^{\infty}\left(0,T;\R^{n}\right)} \bar{x},
\end{equation}
for some $\bar{x} \in AC(0,T;\R^{n})$.
Moreover,
\begin{equation}
\label{eq:ode_before_limit}
x[u_k](t) = x_0 + \int_0^t f_0\left(x[u_k](\lambda)\right) + f_1\left(x[u_k](\lambda)\right) u_k(\lambda) \ \ddl.
\end{equation}
Using \eqref{eq:uniform_cv_xn} and the continuity of $f_0$,
\(
\int_0^t f_0\left(x[u_k](\lambda)\right) \, \ddl \xrightarrow[k \rightarrow +\infty]{} \int_0^t f_0\left(\bar{x}(\lambda)\right) \, \ddl.
\) Then, note that
\begin{align*}
\int_0^t f_1\left(x[u_k](\lambda)\right) u_k(\lambda) - f_1\left(\bar{x}(\lambda)\right)\bar{u}(\lambda) \, \ddl =& \int_0^t \left[ f_1\left(x[u_k](\lambda)\right)-f_1\left(\bar{x}(\lambda)\right) \right] u_k(\lambda) \, \ddl \\
&+ \int_0^t f_1\left(\bar{x}(\lambda)\right)\left( u_k(\lambda) - \bar{u}(\lambda) \right) \, \ddl.
\end{align*}
Using Cauchy-Schwarz and the boundedness of $(u_k)_{k \in \mathbb{N}}$ in $\bfcalH$, we have
\begin{equation*}
\left\| \int_0^t \left( f_1\left(x[u_k](\lambda)\right) - f_1\left(\bar{x}(\lambda)\right) \right) u_k(\lambda) \, \mathrm{d}\lambda \right\|
\leq N \|f_1(x[u_k]) - f_1(\bar{x})\|.
\end{equation*}
Now, as $f_1$ is continuous and $\left(x[u_k]\right)_{k \in \mathbb{N}}$ converges uniformly to $\bar{x}$, the right-hand side of the inequality vanishes as $k \rightarrow +\infty$. It follows that:
\begin{equation}
\label{eq:limit_ode_1}
\int_0^t \left[ f_1\left(x[u_k](\lambda)\right)-f_1\left(\bar{x}(\lambda)\right) \right] u_k(\lambda) \, \ddl \xrightarrow[k \rightarrow +\infty]{} 0.
\end{equation}

Then, let us define for every \(t \in (0,T)\) the operator
\(\tilde{A}_t \in L^2(0,T;\mathbb{R}^{m \times n})\) as well as the mapping \(L_t : \bfcalH \to \mathbb{R}\) defined by
\[
\tilde{A}_t(\lambda) :=
f_1(\bar{x}(\lambda)) \,\mathbf{1}_{[0,t]}(\lambda), \quad L_t(u) := \int_0^T \tilde{A}_t(\lambda)\, u(\lambda)\, d\lambda .
\]
Note that $L_t$ is a linear continuous operator on $\bfcalH$.
Then,
\begin{equation*}
\int_0^t f_1\left(\bar{x}(\lambda)\right)\left( u_k(\lambda) - \bar{u}(\lambda) \right) \, \ddl = L_t ( u_k - \bar{u} ),
\end{equation*}
and as $u_k$ converges weakly to $\bar{u}$,
\(
L_t ( u_k - \bar{u} ) \xrightarrow[k \to +\infty]{} 0.
\)
It follows that
\begin{equation*}
\int_0^t \left[ f_1(x[u_k](\lambda)) u_k(\lambda) - f_1(\bar{x}(\lambda)) \bar{u}(\lambda) \right] \, \ddl \xrightarrow[k \to +\infty]{} 0.
\end{equation*}
Passing to the limit in \eqref{eq:ode_before_limit} then leads to
\(
\bar{x}(t) = x_0 + \int_0^t \left[ f_0(\bar{x}(\lambda)) + f_1(\bar{x}(\lambda)) \bar{u}(\lambda) \right] \, \ddl,
\)
i.e., $\bar{x}=x[\bar{u}]$.
Hence,
\begin{equation}
\int_0^T \ell_1\left(x[u_k](t)\right) \, \ddt + \varphi\left(x[u_k](T)\right) \xrightarrow[k \rightarrow +\infty]{} \int_0^T \ell_1\left(\bar{x}(t)\right) \, \ddt + \varphi\left(\bar{x}(T)\right).
\end{equation}
Combining this limit with \eqref{eq:to_gather_1}, we obtain
\(
\liminf_{k \to +\infty} \phi_1(u_k) \geq \phi_1(\bar{u}).
\)
\\

\noindent \textbf{\textit{Part 2: lower boundedness of $\phi_1$.}} Let $u \in \bfcalH$ and $p \in \R^{m}$. Using the strong convexity of $\ell_0$, there exists some $\alpha >0$ such that
\begin{equation}
\label{eq:lower_bound_phi_1}
\int_0^T \ell_0(u(t)) \ddt
\geq \int_0^T \sclr{\nabla \ell_0 (p),u(t)-p} + \frac{\alpha}{2}\norm{u(t)-p}^2 \ddt + \int_0^T \ell_0(p) \ddt.
\end{equation}
Expending the squared norm, we have
\begin{equation*}
\int_0^T \sclr{\nabla \ell_0 (p),u(t)-p} + \frac{\alpha}{2}\norm{u(t)-p}^2 \ddt = \frac{\alpha}{2} \norm{u}^2 + \int_0^T \sclr{u(t),\nabla \ell_0(p)-2p} \ddt + C_1,
\end{equation*}
where $C_1 \in \R$ depends only on $p$ and $T$. Moreover, thanks to Cauchy-Schwarz inequality:
\begin{equation*}
\int_0^T \sclr{\nabla \ell_0 (p),u(t)-p} + \frac{\alpha}{2}\norm{u(t)-p}^2 \ddt \geq \frac{\alpha}{2} \norm{u}^2 - C_2 \norm{u} + C_1,
\end{equation*}
with $C_2 \in \R$ also depending only on $p,T$. Combining this last inequality and \eqref{eq:lower_bound_phi_1}, together with the fact that $\ell_1,\varphi$ are lower bounded, we deduce that there exists $C, C' \in \R$ such that
\begin{equation}
\label{eq:lower_boundedness_phi}
\phi_1(u) \geq  \frac{\alpha}{2} \norm{u}^2 + C \norm{u} + C'.
\end{equation}
As $\phi_1(u)$ is lower-bounded by a quadratic function of the norm of $u$, the results follows.
\\
\\

\noindent \textbf{\textit{Part 3: local Lipschitz continuity of $\nabla \phi_1$.}} According to classical optimal control theory—see, for instance, \cite[Subsec.~1.6]{hinze2008optimization}— \( \phi_1 \) is continuously differentiable, and for any \( u \in \bfcalH \), its gradient is given by
\begin{equation}
\label{eq:gradient_expression}
\nabla \phi_1(u)(t) = \nabla \ell_0\big(u(t)\big) + f_1\big(x(t)\big)^\top p(t) = H'_u\big(x(t), u(t), p(t)\big) \quad \text{for a.e. } t \in (0,T),
\end{equation}
where \( x = x[u] \) denotes the state trajectory associated with \( u \), and \( p \) is the corresponding adjoint state solving \eqref{eq:dynamics_adjoint_state}. As $\nabla \ell_0$ is assumed Lipschitz continuous, it is straightforward to verify that the mapping
\[
u \in \bfcalH \rightarrow \nabla \ell_0(u) \in \bfcalH
\]
is also (locally) Lipschitz continuous. Thus, it suffices to show that \( f_1(x)^\top p \) is locally Lipschitz with respect to \( u \). To establish this, it is enough to verify that both \( p \) and \( f_1(x) \) are locally Lipschitz, as the scalar product of locally Lipschitz functions is locally Lipschitz. Let us therefore consider two controls \( u_1, u_2 \in \bfcalH \) such that \( \|u_1\|, \|u_2\| \leq M \) for some constant \( M > 0 \).
Let \( x_1, x_2 \) be their associated state variables, solutions to \eqref{eq:dynamics_intro} under Assumption~\ref{assump:gen_OC}, and let \( p_1, p_2 \) be the corresponding adjoint states, solutions to \eqref{eq:dynamics_adjoint_state}. Then, thanks to Lemma~\ref{lemma:verification_assumption_loc_lip}, we have that \(x_1\), \(x_2\), \(p_1\), and \(p_2\) are also uniformly bounded (by $M$ without loss of generality). 
\\
\\
\textit{Step 1: regularity of $f_1(x)$ with respect to the control $u$.}
\\
First, because $f_1$ is locally Lipschitz continuous,
\[
\norm{f_1\left(x_1(t)\right) - f_1\left( x_2(t) \right)} \leq L_{f_1} \norm{x_1(t) - x_2(t)}, \quad \text{for a.e } t \in (0,T),
\]
for some constant $L_{f_1}>0$. For the sake of conciseness, all the equations in the rest of the proof involving the time variable $t$ are written for a.e $t \in (0,T)$. From~Assumption~\ref{assump:gen_OC}, it follows that:
\begin{equation}
\label{eq:proof_smoothness_1}
\norm{\dot{x}_2(t) - \dot{x}_1(t)} \leq L_{f_0} \norm{x_2(t) - x_1(t)} + \norm{ f_1\left(x_2(t)\right)u_2(t) - f_1\left(x_1(t)\right)u_1(t)},
\end{equation}
for some constant $L_{f_0}>0$. Moreover:
\begin{align*}
\norm{ f_1\left(x_2(t)\right)u_2(t) - f_1\left(x_1(t)\right)u_1(t)} &\leq \norm{f_1(x_2(t))}\norm{u_2(t)-u_1(t)}\\
&+ \norm{f_1(x_2(t))-f_1(x_1(t))}\norm{u_1(t)}.
\end{align*}
Moreover, as $f_1$ is locally Lipschitz continuous, we deduce from the last inequality that
\begin{equation*}
\norm{ f_1\left(x_2(t)\right)u_2(t) - f_1\left(x_1(t)\right)u_1(t)} \leq C_{f_1} \norm{u_2(t) - u_1(t)} + L_{f_1} \norm{x_2(t)-x_1(t)} \norm{u_1(t)},
\end{equation*}
for some constant $C_{f_1}>0$. So, from~\eqref{eq:proof_smoothness_1}, we have
\begin{equation*}
\norm{\dot{x}_2(t) - \dot{x}_1(t)} \leq \left(L_{f_0}+L_{f_1} \norm{u_1(t)} \right) \norm{x_2(t)-x_1(t)} + C_{f_1} \norm{u_2(t) - u_1(t)},
\end{equation*}
and applying Grönwall's lemma yields
\begin{equation}
\norm{x_2(t)-x_1(t)} \leq C_{T,M,L_{f_0},L_{f_1},C_{f_1}} \int_0^t \norm{u_2(s)-u_1(s)} \dds \leq C_{T,M,L_{f_0},L_{f_1},C_{f_1}} \norm{u_2 - u_1}_2.
\label{eq:to_use_3}
\end{equation}
We then establish the first part of the result by substituting this inequality into the first one of this step of the proof, squaring both sides, and integrating, which yields
\begin{equation}
\label{eq:lip_state_vs_control}
\norm{f_1(x_2) - f_1(x_1)}^2_2 \leq C_{T,M,L_{f_0},L_{f_1},C_{f_1}}^2 \norm{u_2 - u_1}^2_2.
\end{equation}
\\
\textit{Step 2: regularity of the adjoint state $p$ with respect to the control $u$.}
\\
Note that
\begin{align}
\label{eq:to_combine_0}
\dot{p}_2(t) - \dot{p}_1(t) &= \nabla \ell_1 (x_1(t)) - \nabla \ell_1 (x_2(t)) \\
&+ \ddD f_0 (x_1(t))^\top p_1(t) - \ddD f_0 (x_2(t))^\top p_2(t) \label{eq:to_combine_0_1} \\
&+\sum_{i=1}^n \ddD f_1^i \left(x_1(t)\right)^\top p_1(t) u_1^i(t) -  \ddD f_1^i \left(x_2(t)\right)^\top p_2(t) u_2^i(t). \label{eq:to_combine_0_2}
\end{align}
On the one hand, as $x_1, p_2 \in L^{\infty}(0,T;\R^n)$ and $\ddD f_0$ is locally Lipschitz continuous, it follows that:
\begin{equation}
\label{eq:to_combine_1}
\norm{\ddD f_0 (x_1(t))^\top p_1(t) - \ddD f_0 (x_2(t))^\top p_2(t)} \leq C_{\ddD f_0,p} \left( \norm{ p_1(t)-p_2(t) } + \norm{x_1(t) - x_2(t) } \right),
\end{equation}
for some constant $ C_{\ddD f_0,p}>0$. On the other hand, using similar arguments, for every $i=1,\ldots,n$, we deduce that
\begin{align*}
& \norm{\ddD f_1^i \left(x_1(t)\right)^\top p_1(t) u_1^i(t) - \ddD f_1^i \left(x_2(t)\right)^\top p_2(t) u_2^i(t)} \\
& \qquad \leq C_{Df_1, L_{f_1}} \bigg[ \norm{u_1(t)} \norm{p_1(t)-p_2(t)} + \norm{u_2(t) - u_1(t)} + \norm{x_1(t)-x_2(t)} \norm{u_2(t)} \bigg].
\end{align*}
Combining this last estimate with \eqref{eq:to_combine_1}, and substituting into \eqref{eq:to_combine_0}–\eqref{eq:to_combine_0_1}–\eqref{eq:to_combine_0_2}, it follows that:
\begin{align*}
\|\dot{p}_2(t) - \dot{p}_1(t)\| &\leq C_{D_{f_1}, D_{f_0}, m, p, C_{l_1}, M} \bigg[ \|p_1(t) - p_2(t)\|\left(1 + \|u_1(t)\|\right) \\
&+ \|x_2(t) - x_1(t)\| + \|u_2(t)\| + \|u_2(t) - u_1(t)\| \bigg].
\end{align*}
A direct application of Grönwall Lemma yields:
\begin{align*}
\norm{p_2(T - t) - p_1(T - t)} &\leq C'_{D_{f_1}, D_{f_0}, m, p, C_{l_1}, M} \Bigg[ \norm{\nabla \varphi(x_2(T)) - \nabla \varphi(x_1(T))} \\
&+ \int_t^T \norm{x_2(\tau) - x_1(\tau)}(1 + \norm{u_2(\tau)}) + \norm{u_2(\tau) - u_1(\tau)} \ddtau \Bigg].
\end{align*}
Also,
\begin{equation*}
\norm{\nabla \varphi(x_2(T)) - \nabla \varphi(x_1(T))} \leq L_{\nabla \varphi} \norm{x_2(T) - x_1(T)} \leq C \norm{u_2 - u_1}_2,
\end{equation*}
where the first inequality is due to the Lipschitz continuity of $\nabla \varphi$ and the second one to \eqref{eq:to_use_3}. It follows that:
\begin{align*}
\norm{p_2(t)-p_1(t)} &\leq C_{D_{f_1}, D_{f_0}, m, p, C_{l_1}, M , \nabla \varphi} \left( \norm{x_2 - x_1}_2 + \norm{u_2 - u_1}_2 \right).
\end{align*}
Finally, squaring both sides of the inequality, integrating, and using \eqref{eq:to_use_3}, we obtain the following estimate:
\begin{equation*}
\norm{p_2 - p_1}^2_2 \leq C_{D_{f_1}, D_{f_0}, m, p, C_{l_1}, M \nabla \phi} \norm{u_2 - u_1}^2_2,
\end{equation*}
This proves the Lipschitz continuity of the adjoint state with respect to the control.
\\
\\
\textbf{\textit{Part 4: on $\phi_2$.}} The convexity, properness, lower boundedness of $\phi_2$ as well as the boundedness of its domain are straightforward to verify. We first show that it is also lower semicontinuous. Let us rewrite $\phi_2$ as
\begin{equation*}
\label{eq:phi_2_rewritten_1}
\phi_2(u)=\int_0^T \psi\left(t,u(t)\right) \ddt,
\end{equation*}
with
$\psi : [0,T) \times \R^n \mapsto \R$ defined as
\begin{equation}
\label{eq:phi_2_rewritten_2}
\psi(t,u) := \sum_{j=1}^{J-1} \psi_{j}(u) \mathbf{1}_{[\tau_j , \tau_{j+1})}(t).
\end{equation}
Also, for a.e $t \in (0,T)$, $u \mapsto \psi(t,u)$ is lsc and as for every $j \in \{ 0 , \ldots , J-1 \}$, $\psi_j \in \Gamma_0(\R^n)$ has a bounded domain, it admits a minimizer so that $\inf_{t \in (0,T), u \in \R^n} \psi(t,u) > -\infty$. We deduce that $\phi_2$ is lsc by a direct application of \cite[Lemma 9.1]{carlier2021classical}. Finally, since \(\phi_2 \) is also convex, we deduce that it is weakly lower-semicontinuous; see, for instance, \cite[Thm.~3.8]{carlier2021classical}.

\subsection{ Proof of Lemma~\ref{lemma:pmp_main_problem}}
\label{appendix:proof_pmp_main_problem}
A direct consequence of the definition of the regular subdifferential—see also \cite[Ex.~8.8]{rockafellar1998variational}—together with the expression of the gradient of $\phi_1$ as in~\eqref{eq:gradient_expression} yields
\begin{equation}
\label{eq:necessary_condition_local_optimality_2}
\hat{\partial} \phi(u) = \nabla \phi_1(u) + \hat{\partial} \phi_2(u).
\end{equation}
As $\phi_2 \in \Gamma_0(H)$, thanks to \cite[Thm. 16.3]{clason2017nonsmooth}, $\hat{\partial} \phi_2 (u) = \partial \phi_2(u)$.
Combining \eqref{eq:necessary_condition_local_optimality_2} and \eqref{eq:gradient_expression}, it follows that the necessary condition for local optimality (see \cite[Thm.~16.2]{clason2017nonsmooth})
\(
0 \in \hat{\partial} \phi(u)
\)
can be equivalently written as
\(
-\nabla \phi_1(u) \in \partial \phi_2(u).
\)
Next, we show that
\begin{equation}
\label{eq:subdiff_phi_2_charact}
\partial \phi_2(u) = \left\{ v \in \bfcalH \mid \forall j \in \{0,\ldots,J-1\},\ \text{ for a.e. } t \in (\tau_j,\tau_{j+1}):\ v(t) \in \partial \psi_j\big(u(t)\big) \right\}.
\end{equation}
It is straightforward to verify that if, for almost every \( t \in (\tau_j, \tau_{j+1}) \), we have \( v(t) \in \partial \psi_j\big(u(t)\big) \), then \( v \in \partial \phi_2(u) \).
Conversely, if \( v \in \partial \phi_2(u) \), then by the definition of the subdifferential, for all \( u' \in \bfcalH \) we have
\begin{equation*}
\phi_2(u') \geq \phi_2(u) + \sclr{v, u' - u}_\bfcalH,
\end{equation*}
so that
\begin{equation}
\label{eq:characterization_prox}
u \in \argmin_{u' \in H} \ \phi_2(u') + \sclr{v , u' - u}_\bfcalH.
\end{equation}
Recall the definition of $\psi$ as in \eqref{eq:phi_2_rewritten_2}.
From \eqref{eq:characterization_prox}, it follows that
\begin{equation*}
u \in \argmin_{u' \in H} \ \int_0^T \psi\left(t,u'(t)\right) - \sclr{v(t),u'(t)} \ddt.
\end{equation*}
Then, the mapping \( (t,u) \mapsto \psi(t,u) - \sclr{v(t),u} \) is a normal integrand in the sense of \cite[Def.~14.27]{rockafellar1998variational}. We justify this claim in the next paragraph. An application of \cite[Thm.~14.60]{rockafellar1998variational} yields the following result:
\begin{equation*}
u(t) \in \argmin_{u \in \mathbb{R}^n} \left\{ \psi(t,u) - \sclr{v(t),u} \right\} \quad \text{for a.e. } t \in (0,T).
\end{equation*}
It follows that \( v(t) \in \partial \psi\left(t,u(t)\right) \), which in turn implies that \( v(t) \in \partial \psi_j\left(u(t)\right) \) for almost every \( t \in (\tau_j, \tau_{j+1}) \) and every \( j \in \{0,\ldots,J-1\} \). This shows, in particular, that \( v \) belongs to the set on the right-hand side of the equality in \eqref{eq:subdiff_phi_2_charact}, and the result of the lemma follows.

We now justify the required assumption of the invoked theorem, \cite[Thm.~14.60]{rockafellar1998variational}. Note that for almost every \( t \in (0,T) \), the function \( \psi(t,\cdot) - \langle v(t), \cdot \rangle \) is convex and lower semicontinuous, and for every \( u \in \mathbb{R}^n \), the mapping \( t \mapsto \psi(t,u) - \langle v(t), u \rangle \) is measurable. Moreover, recalling the expression of \( \psi \) in \eqref{eq:phi_2_rewritten_2}, for every \( j \in \{1, \ldots, J-1\} \) and \( t \in (\tau_j,\tau_{j+1}) \), one can find a countable dense subset \( D_j \) of the domain of \( \psi(t,\cdot) = \psi_j \), which is also the domain of \( \psi(t,\cdot) - \sclr{v(t),\cdot} \), and construct a countable family \( U_j \subseteq \bfcalH \) of constant functions from \( [0,T] \) to \( D_j \). Then, the countable family of functions \( \bigcup_{j=1}^{J-1} U_j \) satisfies the requirements of \cite[Prop.~14.39]{rockafellar1998variational}, and together with the aforementioned properties, the mapping \( (t,u) \mapsto \psi(t,u) - \sclr{v(t),u} \) is indeed a normal integrand. Similar arguments can be found in the proof of \cite[Lemma~1]{rockafellar1968integrals}.

\subsection{ Proof of Lemma \ref{lemma:stability_condition}}
\label{appendix:proof_stability_condition}
Under the stated assumptions, the set $N \cap C$ is closed and convex, so the projection $\proj{N \cap C}$ is well defined.
Also, any $x \in H$ admits the orthogonal decomposition
\[
x = x_{N} + x_{N^\perp}, \qquad x_{N} \in N,\; x_{N^\perp} \in N^\perp.
\]
By the variational characterization of projections onto closed convex sets
\cite[Theorem~1.24]{carlier2021classical}, we have
\[
\langle x - \proj{N \cap C}(x),\, y - \proj{N \cap C}(x) \rangle \le 0,
\qquad \forall\, y \in N \cap C.
\]
Expanding the inner product yields
\begin{align*}
\langle x - \proj{N \cap C}(x),\, y - \proj{N \cap C}(x) \rangle
&= \langle x_{N} - \proj{N \cap C}(x),\, y - \proj{N \cap C}(x) \rangle \\ 
&+ \langle x_{N^\perp},\, y - \proj{N \cap C}(x) \rangle.
\end{align*}
The second term vanishes because $y - \proj{N \cap C}(x) \in N$.
To prove \eqref{eq:composition_proj}, it suffices to show that
\begin{equation}
\label{eq:co_composition}
\langle x_{N} - \proj{C} \circ \proj{N}(x),\, y - \proj{C} \circ \proj{N}(x) \rangle \le 0,
\qquad \forall\, y \in N \cap C,
\end{equation}
and that $\proj{C} \circ \proj{N}(x) \in N \cap C$. Since $\proj{N}(x) = x_{N}$ and $y \in C$, we obtain
\[
\langle x_{N} - \proj{C}(x_{N}),\, y - \proj{C}(x_{N}) \rangle \le 0,
\]
which is precisely \eqref{eq:co_composition}.
Finally, the stability condition \eqref{eq:stability_condition} ensures that
$\proj{C} \circ \proj{N}(x) \in N$, and by definition it also lies in $C$.
Thus $\proj{C} \circ \proj{N}(x) \in N \cap C$, completing the proof.

\bibliographystyle{plain}   
\bibliography{samplebib}

\end{document}